\documentclass[a4paper, 11pt]{amsart}   
\usepackage{mathptmx, amssymb,amscd,latexsym, eulervm}   
\usepackage{amsmath}
\usepackage{amsthm}
\usepackage{mathdots}
\usepackage[pagebackref,colorlinks=true,linkcolor=blue,urlcolor=blue]{hyperref}
\usepackage{color}
\usepackage[onehalfspacing]{setspace}
\usepackage{tabularx}
\usepackage{amsfonts}
\usepackage{paralist}
\usepackage{aliascnt}
\usepackage[initials, lite]{amsrefs}
\usepackage{amscd}
\usepackage{blkarray}
\usepackage{mathbbol}
\usepackage{setspace}
\usepackage[inner=2.3cm,outer=2.3cm, bottom=3.2cm]{geometry}
\usepackage{tikz, tikz-cd}

\usepackage{tikz}
\usetikzlibrary{matrix}
\usetikzlibrary{arrows,calc}
\allowdisplaybreaks

\BibSpec{collection.article}{%
	+{}  {\PrintAuthors}                {author}
	+{,} { \textit}                     {title}
	+{.} { }                            {part}
	+{:} { \textit}                     {subtitle}
	+{,} { \PrintContributions}         {contribution}
	+{,} { \PrintConference}            {conference}
	+{}  {\PrintBook}                   {book}
	+{,} { }                            {booktitle}
	+{,} { }                            {series}
	+{, vol.} { }                            {volume}
	+{,} { }                            {publisher}
	+{,} { \PrintDateB}                 {date}
	+{,} { pp.~}                        {pages}
	+{,} { }                            {status}
	+{,} { \PrintDOI}                   {doi}
	+{,} { available at \eprint}        {eprint}
	+{}  { \parenthesize}               {language}
	+{}  { \PrintTranslation}           {translation}
	+{;} { \PrintReprint}               {reprint}
	+{.} { }                            {note}
	+{.} {}                             {transition}
	+{}  {\SentenceSpace \PrintReviews} {review}
}

\newcommand{\RR}{\mathbb{R}}
\newcommand{\kk}{\mathbb{k}}

\newcommand{\bmm}{\mathbb{M}}

\newcommand{\NN}{\normalfont\mathbb{N}}
\newcommand{\ZZ}{{\normalfont\mathbb{Z}}}

\newcommand{\PP}{{\normalfont\mathbb{P}}}

\newcommand{\dd}{{\normalfont\mathbf{d}}}

\newcommand{\mm}{{\normalfont\mathfrak{m}}}

\newcommand{\SSS}{\mathcal{S}}

\newcommand{\pp}{{\normalfont\mathfrak{p}}}

\newcommand{\bn}{{\normalfont\mathbf{n}}}

\newcommand{\bx}{{\normalfont\mathbf{x}}}

\newcommand{\bm}{{\normalfont\mathbf{m}}}
\newcommand{\ttt}{{\normalfont\mathbf{t}}}

\newcommand{\fJ}{{\mathfrak{J}}}

\newcommand{\conv}{\normalfont\text{conv}}

\newcommand{\Quot}{\normalfont\text{Quot}}

\newcommand{\II}{\mathbb{I}}

\newcommand{\Rees}{\mathcal{R}}

\newcommand{\ee}{{\normalfont\mathbf{e}}}

\newcommand{\Pic}{{\normalfont\text{Pic}}}

\newcommand{\JJ}{\mathbb{J}}

\newcommand{\bJ}{\mathbf{J}}

\newcommand{\bK}{\mathbf{K}}
\newcommand{\OO}{\mathcal{O}}

\newcommand{\FF}{\mathbb{F}}

\newcommand{\HH}{\normalfont\text{H}}

\newcommand{\Spec}{\normalfont\text{Spec}}

\newcommand{\Vol}{{\normalfont\text{Vol}}}
\newcommand{\MV}{{\normalfont\text{MV}}}
\newcommand{\ind}{{\normalfont\text{ind}}}
\newcommand{\Con}{{\normalfont\text{Con}}}

\def\f0{\mathbf{0}}

\def\fn{\mathbf{n}}
\def\fm{\mathbf{m}}

\def\ft{\mathbf{t}}
\def\fd{\mathbf{d}}

\def\ft{\mathbf{t}}

\def\ls{\leqslant}
\def\gs{\geqslant}

\def\*{{\color{red}\blacksquare}}

\newtheorem{theorem}{Theorem}[section]

\newtheorem{headthm}{Theorem}

\newaliascnt{headcor}{headthm}

\aliascntresetthe{headcor}

\newaliascnt{headthmdef}{headthm}

\aliascntresetthe{headthmdef}

\newaliascnt{headconj}{headthm}

\aliascntresetthe{headconj}

\newaliascnt{corollary}{theorem}
\newtheorem{corollary}[corollary]{Corollary}
\aliascntresetthe{corollary}

\newaliascnt{lemma}{theorem}
\newtheorem{lemma}[lemma]{Lemma}
\aliascntresetthe{lemma}

\newaliascnt{conjecture}{theorem}

\aliascntresetthe{conjecture}

\newaliascnt{proposition}{theorem}
\newtheorem{proposition}[proposition]{Proposition}
\aliascntresetthe{proposition}

\theoremstyle{definition}
\newaliascnt{definition}{theorem}
\newtheorem{definition}[definition]{Definition}
\aliascntresetthe{definition}

\newaliascnt{notation}{theorem}

\aliascntresetthe{notation}

\newaliascnt{example}{theorem}
\newtheorem{example}[example]{Example}
\aliascntresetthe{example}

\newaliascnt{examples}{theorem}

\aliascntresetthe{examples}

\newaliascnt{remark}{theorem}
\newtheorem{remark}[remark]{Remark}
\aliascntresetthe{remark}

\newaliascnt{problem}{theorem}

\aliascntresetthe{problem}

\newaliascnt{question}{theorem}

\aliascntresetthe{question}

\newaliascnt{convention}{theorem}

\aliascntresetthe{convention}

\newaliascnt{construction}{theorem}

\aliascntresetthe{construction}

\newaliascnt{setup}{theorem}
\newtheorem{setup}[setup]{Setup}
\aliascntresetthe{setup}

\newaliascnt{algorithm}{theorem}

\aliascntresetthe{algorithm}

\newaliascnt{observation}{theorem}

\aliascntresetthe{observation}

\newaliascnt{defprop}{theorem}

\aliascntresetthe{defprop}

\def\equationautorefname~#1\null{(#1)\null}
\def\sectionautorefname~#1\null{Section #1\null}
\def\subsectionautorefname~#1\null{\S #1\null}
\def\trdeg{{\rm trdeg}}

\begin{document}
	
	\title{Multigraded algebras and multigraded linear series}

	\author{Yairon Cid-Ruiz}
	\address[Cid-Ruiz]{Department of Mathematics, North Carolina State University, Raleigh, NC 27695, USA}
	\email{ycidrui@ncsu.edu}
	
	\author{Fatemeh Mohammadi}
	\address[Mohammadi]{Department of Mathematics, KU Leuven, Celestijnenlaan 200B, Leuven, Belgium  and Department of Computer Science, KU Leuven, Celestijnenlaan 200A, 3001 Leuven, Belgium}
	\email{fatemeh.mohammadi@kuleuven.be}

	\author{Leonid Monin}
	\address[Monin]
	{Institute of Mathematics, Ecole Polytechnique F\'ed\'erale de Lausanne, Switzerland}
	\email{leonid.monin@epfl.ch}
	
	
	\date{\today}
	
	\begin{abstract}
	This paper is devoted to the study of multigraded algebras and multigraded linear series. 
	For an $\NN^s$-graded algebra $A$, we define and study its volume function $F_A:\NN_+^s\to \RR$, which computes the asymptotics of the Hilbert function of $A$. 
	We relate the volume function $F_A$ to the volume of the fibers of the global Newton-Okounkov body $\Delta(A)$ of $A$.
	Unlike the classical case of standard multigraded algebras, the volume function $F_A$ is not a polynomial in general. However, in the case when the algebra $A$ has a decomposable grading, we show that the volume function $F_A$ is a polynomial with non-negative coefficients. 
	We then define mixed multiplicities in this case and provide a full characterization for their positivity. 
	Furthermore, we apply our results on multigraded algebras to multigraded linear series.
	
	Our work recovers and unifies recent developments on mixed multiplicities. 
	In particular, we recover results on the existence of mixed multiplicities for (not necessarily Noetherian) graded families of ideals and on the positivity of the multidegrees of multiprojective varieties.
	\end{abstract}
	
	\maketitle
	
	{
		\hypersetup{linkcolor=black}
		\tableofcontents
	}

	\section{Introduction}
	
	The main goal of this paper is to study and extend the current theory of \emph{graded algebras of almost integral type} and \emph{graded linear series} to \emph{multigraded algebras of almost integral type} and \emph{multigraded linear series}, respectively. The asymptotic behavior of graded algebras of almost integral type and graded linear series was extensively studied in  the foundational papers of Kaveh and Khovanskii \cite{KAVEH_KHOVANSKII} and of Lazarsfeld and Musta\c{t}\u{a} \cite{lazarsfeld09}. 
	Following ideas from seminal works of Okounkov \cite{Ok1, Ok2}, the authors of \cite{KAVEH_KHOVANSKII} and \cite{lazarsfeld09} associated a convex body $\Delta(A)$ to a graded algebra $A$ equipped with a valuation with one-dimensional leaves. The convex body $\Delta(A)$ is called the \emph{Newton-Okounkov body} of the algebra $A$. 
	One of the main results of \cite{KAVEH_KHOVANSKII, lazarsfeld09} relates the asymptotic growth of the Hilbert function of a graded algebra of almost integral type with the corresponding Newton-Okounkov body. 
	
	More precisely, let $\kk$ be an algebraically closed field and $\FF$ be a field containing $\kk$. 
	Recall that a graded $\kk$-subalgebra $A\subset \FF[t]$ of the polynomial ring in one variable is of \emph{integral type} if $A$ is finitely generated over $\kk$ and is a finitely generated module over the subalgebra generated by $[A]_1$. 
	A graded $\kk$-algebra $A \subset \FF[t]$ is of \emph{almost integral type} if $A \subset B \subset \FF[t]$, where $B$ is a graded $\kk$-algebra of integral type. Then, one has the following theorem.
	
	\begin{theorem}[\cite{KAVEH_KHOVANSKII, lazarsfeld09}]\label{thm-kkml}
	Let $A\subset \FF[t]$ be a graded $\kk$-algebra of almost integral type, and let $q=\dim_\RR(\Delta(A))$.
	Then the Hilbert function of $A$ has a polynomial growth. 
	Moreover, the main term of asymptotics has degree $q$ and its coefficient is given by the volume of $\Delta(A)$.
	\end{theorem}

	We generalize \autoref{thm-kkml} to a far-reaching general case of algebras of almost integral type. We consider an arbitrary field $\kk$ (not necessarily algebraically closed) and we replace the field $\FF$ with any reduced $\kk$-algebra $R$. We also relate the dimension of the Newton-Okounkov body of $A\subset R[t]$ to the Krull dimension of $A$.
	
	\begin{headthm}[\autoref{thm_limit_graded}]
	    \label{thm_A}
	    Let $\kk$ be a field and $R$ be a reduced $\kk$-algebra.
		Let $A \subset R[t]$ be a graded $\kk$-algebra of almost integral type. 
		Suppose $d = \dim(A) > 0$.
		Then, there exists an integer $m > 0$ such that the limit 
		$$
		\lim_{n\to \infty} \frac{\dim_{\kk}\left([A]_{nm}\right)}{n^{d-1}}  \;\in\; \RR_{+}
		$$
		exists and it is a positive real number.   
	\end{headthm}
	
	We prove this result by analyzing the graded algebras equipped with a valuation having leaves of bounded dimension. 
	In particular, we show that any algebra $A \subset R[t]$ of almost integral type has a valuation with bounded leaves (\autoref{prop_bounded_val}). 
	Our treatment of singly graded algebras also profited from the works of Cutkosky \cite{cutkosky2013, cutkosky2014}.
	
	\smallskip
	
	We now focus on our developments of the multigraded case of algebras and linear series.

	\addtocontents{toc}{\protect\setcounter{tocdepth}{1}}
	\subsection{Multigraded algebras of almost integral type}
	Here we describe our main results regarding arbitrary multigraded algebras of almost integral type.
	
	Let $\kk$ be an arbitrary field and $R$ be a $\kk$-domain. 
	Our results in principle could cover the case when $R$ is an arbitrary 
	reduced ring (see the general reductions used in \autoref{sect_single_grad}), however, in doing so the notation would be quite cumbersome in the notion of volume function that we define below.
	
	First, we introduce our main object of study, which provides the multigraded extension of graded algebras of almost integral type (as introduced in \cite{KAVEH_KHOVANSKII}).
	Let $t_1,\ldots,t_s$ be new variables  over $R$ and consider $R[t_1,\ldots,t_s]$ as a standard $\NN^s$-graded polynomial ring where $\deg(t_i) = \ee_i  \in \NN^s$ and $\ee_i$ denotes the $i$-th elementary vector $(0,\ldots, 1,\dots, 0)$.
	We have the following notions:
	
	\begin{enumerate}[(i)]
			\item An $\NN^s$-graded $\kk$-algebra $A = \bigoplus_{\bn \in \NN^s} [A]_{\bn} \subset R[t_1,\ldots,t_s]$ is said to be of \emph{integral type} if $A$ is finitely generated over $\kk$ and is a finite module over the subalgebra generated by $[A]_{\ee_1}, [A]_{\ee_2},\ldots,[A]_{\ee_s}$.
			\item An $\NN^s$-graded $\kk$-algebra $A = \bigoplus_{\bn \in \NN^s} [A]_{\bn} \subset R[t_1,\ldots,t_s]$ is said to be of \emph{almost integral type} if $A \subset B \subset R[t_1,\ldots,t_s]$, where $B$ is an $\NN^s$-graded algebra of integral type.
	\end{enumerate}

	Let $A \subset R[t_1,\ldots,t_s]$ be an $\NN^s$-graded algebra of almost integral type.
	To simplify notation, we also need to assume that $[A]_{\ee_i} \neq 0$ for all $1 \le i \le s$.  
	One has that the Krull dimension of $A$ is finite (see \autoref{thm_dim_subfinite_alg}), i.e.,~$\dim(A) < \infty$.
	Let $d = \dim(A)$ and $q = d -s$.
	Our main focus is on the \emph{volume function} of $A$, which is defined as 
	\begin{equation}
		\label{eq_volume_funct}
		F_A : \ZZ_+^s \rightarrow \RR, \quad\quad F_A(n_1,\ldots,n_s) \;= \; \lim_{n \to \infty} \frac{\dim_{\kk}\Big([A]_{(nn_1,\ldots,nn_s)}\Big)}{n^q}.
	\end{equation}
	As a direct consequence of \autoref{thm_limit_graded_domain}, we obtain that the limit defining the volume function $F_A$ always exists.
	Note that,  when $A$ is a standard $\NN^s$-graded algebra then the volume function $F_A$ encodes the mixed multiplicities of $A$ (see \autoref{rem_standard_graded}).
	
	In a similar fashion to \cite{lazarsfeld09}, we relate the volume function $F_A$ of $A$ to a global Newton-Okounkov body that we define below. 
	For that, we can safely assume that $R$ is finitely generated over $\kk$ (see \autoref{rem_finite_gen}) and that there exists a valuation $\nu : \Quot(R) \rightarrow \ZZ^r$ with certain good properties (see \autoref{prop_bounded_val}).
	Of particular importance is the fact that the valuation $\nu$ has leaves of bounded dimension. 
	Let $\Gamma_A$ be the \emph{valued semigroup} of the $\NN^s$-graded algebra $A$:
	\[
	\Gamma_A \;=\; \big\{(\nu(a),\bm)\in  \ZZ^{r}\times\NN^{s} \;\mid\; 0 \neq  a\in [A]_\bm \big\}.
	\]
	We define 
	$$ 
	\Delta(A) \;=\;  \Con(\Gamma_A) \subset \RR^r \times \RR_{\geq 0}^s
	$$ 
	to be the closed convex cone generated by $\Gamma_A$ and we call it the \emph{global Newton-Okounkov body} of $A$.
	We denote the index of $A$ by $\ind(A)$ and the maximal dimension of leaves of $A$ by $\ell_A$ (these invariants refer to a uniform behavior of the Veronese subalgebras $A^{(\bn)} = \bigoplus_{n = 0}^\infty [A]_{n\bn}$; for more details, see  \autoref{def_uniform_A}).
	One has the diagram below
	\begin{center}
		\begin{tikzpicture}
			\matrix (m) [matrix of math nodes,row sep=1.8em,column sep=2.5em,minimum width=2.5em, text height=1.5ex, text depth=0.25ex]
			{
				\Delta(A) &               & \RR^r \times \RR_{\ge 0}^s & \\
				&  \RR^r & & \RR_{\ge 0}^s\\				
			};
			\path[-stealth]
			(m-1-3) edge node [above]	{$\pi_1$} (m-2-2)
			(m-1-3) edge node [above]  {$\pi_2$} (m-2-4)
			;	
			\draw[right hook->] (m-1-1)--(m-1-3);			
		\end{tikzpicture}	
	\end{center}
	where $\pi_1  : \RR^{r} \times \RR_{\ge 0}^s \rightarrow \RR^r$ and $\pi_2  : \RR^{r} \times \RR_{\ge 0}^s \rightarrow \RR_{\ge 0}^s$ denote the natural projections.
	We denote the \emph{fiber} of the global Newton-Okounkov body $\Delta(A)$ over $x\in \RR_{\geq 0}^s$ by $\Delta(A)_x = \Delta(A) \cap \pi_2^{-1}(x)$.
	
	The following theorem contains our main results regarding the relations between the volume function $F_A$ and the global Newton-Okounkov body $\Delta(A)$.
	
	\begin{headthm}[{\autoref{thm-globalNO}, \autoref{cor-global}}]\label{head-globalNO}
		Under the notations above the following results hold:
		\begin{enumerate}[(i)]
			\item The fiber $\Delta(A)_{\bn}$ of the global Newton-Okounkov body $\Delta(A) \subset \RR^r \times \RR_{\ge 0}^s$  coincides with the Newton-Okounkov body $\Delta(A^{(\bn)})$ for each $\bn\in \ZZ_+^s$, that is:
			$$
			\pi_1\left(\Delta(A)_\bn\right) \times \{1\} \;= \; \pi_1 \left(\Delta(A) \cap \pi_2^{-1}(\bn) \right) \times \{1\} \;=\; \Delta(A^{(\bn)}) \; \subset \; \RR^r \times \{1\}.
			$$
			\item There exists a unique continuous homogeneous function of degree $q$ extending the volume function $F_A(\bn)$ defined in \autoref{eq_volume_funct} to the positive orthant $\RR_{\geq 0}^s$.
			This function is given by 
			$$
			F_A : \RR_{\geq 0}^s\to \RR, \qquad x \mapsto \ell_A \cdot \frac{\Vol_q(\Delta(A)_x)}{\ind(A)}.
			$$
			Moreover, the function $F_A$ is log-concave:
			$$
			F_A(x+y)^{\frac1q} \; \geq \; F_A(x)^{\frac1q}+F_A(y)^{\frac1q}\quad\text{for all } x,y\in \RR_{\geq 0}^s.
			$$
		\end{enumerate}
	\end{headthm}
	
	We illustrate the concrete computation of the volume functions in  \autoref{examp_non_poly} and \autoref{examp_funct_min}. 
	In particular, we show that in general, part \emph{(ii)} of \autoref{head-globalNO} is the strongest result that can be obtained in this setting: for any non-negative,  homogeneous of degree $1$, concave function $f:\RR^s_{\geq 0} \to \RR_{\geq0}$ we construct an $\NN^s$-graded algebra $A$ whose volume function is $f$ (see \autoref{ex-universal-volume}).

	\subsection{Algebras with decomposable grading}
	An important question is to determine the following:
	\begin{itemize}
		\item	\emph{When is the volume function $F_A$  a polynomial?}
	\end{itemize}
	By mimicking the classical case of standard multigraded algebras, the existence of such a polynomial yields a natural notion of mixed multiplicities for the algebra $A$ (see \autoref{rem_standard_graded}). 
	It turns out that the volume function $F_A$ is not always a polynomial (see, e.g.,~\autoref{examp_non_poly}, \autoref{ex-universal-volume} and \autoref{examp_funct_min}). 
However, we show that the volume function is a polynomial for the family of multigraded algebras with decomposable gradings.

An $\NN^s$-graded algebra $A$ is said to have a \emph{decomposable grading} if the equality
		$$
		[A]_{(n_1,n_2,\ldots,n_s)} = [A]_{n_1\ee_1} \cdot [A]_{n_2\ee_2}  \cdot \, \cdots \, \cdot [A]_{n_s\ee_s}  
		$$
	holds for each $(n_1,n_2,\ldots,n_s) \in \NN^s$.

	Now, we additionally assume that the $\NN^s$-graded algebra $A \subset R[t_1,\ldots,t_s]$ of almost integral type has a decomposable grading.
	For each $p \ge 1$, let 
	$
	\widehat{A}_{[p]} \;=\; \kk\left[ [A]_{p\ee_1},\ldots, [A]_{p\ee_s} \right] \subset A
	$
	be the $\NN^s$-graded algebra generated by $[A]_{p\ee_1},\ldots,[A]_{p\ee_s}$, and denote by $\widetilde{A}_{[p]} = \bigoplus_{\bn \in \NN^s} \left[\widehat{A}_{[p]}\right]_{p\bn}$ the standard $\NN^s$-graded algebra obtained by regrading $\widehat{A}_{[p]}$.
	
	The following theorem deals with the case of algebras with a decomposable grading.
	
	\begin{headthm}[\autoref{thm_poly_decomp}]
		\label{thm_C}
		Assume the notations above with $A$ having a decomposable grading.
		Then, there exists a homogeneous polynomial $G_A(\bn) \in \RR[n_1,\ldots,n_s]$ of degree $q$ with non-negative real coefficients such that 
		$$
		F_A(\bn) \; = \; G_A(\bn) \quad \text{ for all } \quad  \bn \in \ZZ_+^s.
		$$
		Additionally, we have 
		$$
		G_A(\bn)  \; = \; \lim_{p \to \infty} \frac{G_{\widetilde{A}_{[p]}}(\bn)}{p^q} \; = \; \sup_{p \in \ZZ_+} \frac{G_{\widetilde{A}_{[p]}}(\bn)}{p^q} \quad \text{ for all } \quad  \bn \in \ZZ_+^s.
		$$
	\end{headthm}
		We illustrate \autoref{thm_C} with an example of the Cox ring of a full flag variety (\autoref{examp_cox_ring}). 

	Moreover, we can write the polynomial $G_A$ of \autoref{thm_C} as follows
	$$
	G_A(\bn) \;=\; \sum_{|\dd| = q} \frac{1}{\dd!}\, e(\dd;A)\, \bn^\dd.
	$$
	Then, for each  $\dd = (d_1,\ldots,d_s) \in \NN^s$ with $|\dd| = q$, we define the non-negative real number 
	$
	e(\dd;A) \ge 0
	$ 
	to be the {\it mixed multiplicity of type $\dd$ of $A$}.

	In \autoref{cor_vol_mult}, we provide an extension into a multigraded setting of the Fujita approximation theorem for graded algebras given in \cite[Theorem 2.35]{KAVEH_KHOVANSKII}.
	We show the following equalities 
		$$
	e(\dd;A) 
	\; = \; 
	\lim_{p \to \infty}  \frac{e\big(\dd; \widetilde{A}_{[p]}\big)}{p^q}
	\; = \; 
	\sup_{p \in \ZZ_+}  \frac{e\big(\dd; \widetilde{A}_{[p]}\big)}{p^q}.
	$$
	Furthermore, in \autoref{thm_postivity_decomp}, we provide a full characterization for the positivity of the mixed multiplicities $e(\dd;A)$ of $A$.
	
	In our findings, we got that a number of interesting applications follow rather easily by studying certain multigraded algebras with decomposable grading.
	The list of applications includes the following: 
	\begin{enumerate}
		\item In \autoref{sect_appl_ideals}, we recover some results from \cite{MIXED_MULT_GRAD_FAM, MIXED_VOL_MONOM, cutkosky2019} by showing the existence of mixed multiplicities for (not necessarily Noetherian) graded families of ideals.
		\item In \autoref{subsect_equigen_ideals}, we provide a full characterization for the positivity of the mixed multiplicities of graded families of equally generated ideals.
		\item In \autoref{subsect_mixed_vol}, we recover a classical characterization for the positivity of the mixed volumes of convex bodies.
	\end{enumerate}
	
	\subsection{Multigraded linear series}
Finally, we are interested in the notion of multigraded linear series. 
	Let $\kk$ be an arbitrary field and $X$ be a proper variety over $\kk$.
	Let $D_1,\ldots,D_s$ be a sequence of Cartier divisors on $X$, and consider the corresponding section ring
	$$
	\SSS(D_1,\ldots,D_s) \; = \; \bigoplus_{(n_1,\ldots,n_s) \in \NN^s} \HH^0\big(X, \OO(n_1D_1+\cdots+n_sD_s)\big).
	$$
	A \emph{multigraded linear series} associated to the divisors $D_1,\ldots,D_s$ is an $\NN^s$-graded $\kk$-subalgebra $W$ of the section ring $\SSS(D_1,\ldots,D_s)$. 
In particular, $W$ is 
of almost integral type (see \autoref{prop_sect_ring}).
	The \emph{Kodaira-Itaka dimension} of $W$ is 
	given by 
	$$
	\kappa(W) \; = \; \dim(W) - s,
	$$
	where as before $\dim(W)$ denotes the Krull dimension of $W$.
	As simple consequence of our developments for multigraded algebras, we have the following theorem for multigraded linear series (which we enunciate below for the sake of completeness).
	
	\begin{headthm}[\autoref{thm-linear-series}]
	\label{thm_D}
		Under the above notations, let $W \subset \SSS(D_1,\ldots,D_s)$ be a multigraded linear series and suppose that $[W]_{\ee_i} \neq 0$ for all $1 \le i \le s$.				
		Then, the following statements hold: 
		\begin{enumerate}[(i)]
			\item The volume function 
			$$
			F_W(\bn) \; = \; \lim_{n \to \infty} \frac{\dim_{\kk}\big([W]_{n\bn}\big)}{n^{\kappa(W)}}
			$$
			of $W$ is well-defined for all $\bn \in \ZZ_+^s$.
			\item There exists a unique continuous function that is homogeneous of degree $\kappa(W)$ and log-concave and that extends the volume function $F_W(\bn)$ 
			to the positive orthant $\RR_{\geq 0}^s$.
			This function is given by 
			$$
			F_W \, :\, \RR_{\geq 0}^s\to \RR, \qquad x \,\mapsto\, \ell_W \cdot \frac{\Vol_{\kappa(W)}\big(\Delta(W)_x\big)}{\ind(W)}.
			$$
			\item If $W$ has a decomposable grading, then there exists a homogeneous polynomial $G_W(\bn) \in \RR[n_1,\ldots,n_s]$ of degree $\kappa(W)$ with non-negative real coefficients such that 
			$$
			F_W(\bn)  \;=\;  G_W(\bn)
			$$  for all $\bn \in \ZZ_+^s$.
			\end{enumerate}
	\end{headthm}

	Finally, in \autoref{thm_Kodaira_multigrad}, we express the mixed multiplicities of a multigraded linear series with decomposable grading in terms of the multidegrees of the multiprojective varieties obtained as the image of certain Kodaira rational maps.
	
	\subsection{Organization of the paper}
	
	The basic outline of this paper is as follows. 
	In \autoref{sect_notations}, we recall some important results and fix some notations.
	In \autoref{sect_single_grad}, we deal with singly graded algebras and we prove \autoref{thm_A}.
	In \autoref{sect_mult_grad}, we begin our study of multigraded algebras and we prove \autoref{head-globalNO}.
	Our treatment of multigraded algebras with decomposable grading is made in \autoref{sect_decomp_grad}, where we show \autoref{thm_C}.
	In \autoref{sect_appl_ideals}, we obtain some applications for the mixed multiplicities of graded families of ideals and for the mixed volumes of convex bodies.
	Finally, in \autoref{section_mult_lin_series}, we apply our results on multigraded algebras to multigraded linear series and we prove \autoref{thm_D}.

	\addtocontents{toc}{\protect\setcounter{tocdepth}{2}}
	\section{Notation and preliminaries}
	\label{sect_notations}
	
	In this preparatory section, we fix our notation and recall some important results to be used throughout the paper.
	We denote the set of non-negative integers by $\NN = \{ 0,1,2,\ldots \}$ and the set of positive integers by $\ZZ_+ = \{1,2,\ldots\}$.
	Let $d \ge 1$.
	For a vector $\fn=(n_1,\ldots, n_d)\in \NN^d$ we denote by $|\fn|$ the sum of its entries. We also denote by $\ee_i\in \NN^{d}$ the $i$-th elementary vector $(0,\ldots, 1,\dots, 0)$. 
	For $\fn=(n_1,\ldots, n_d)$ and $\fm=(m_1,\ldots, m_d)$ in $\NN^d$ we write $\fn\gs \fm$ if $n_i\gs m_i$ for every $1\ls i\ls d$. 
	The null vector $(0,\ldots,0) \in \NN^d$ is denoted by $\mathbf{0} \in \NN^d$. 
	Moreover, we write $\fn\gg \mathbf{0}$ if $n_i\gs 0$ for every $1\ls i\ls d$. 
	We also use the abbreviations $\fn!=n_1!\cdots n_d!$ and $\fn^\fm=n_1^{m_1}\cdots n_d^{m_d}$.
	
	\smallskip
	
	We first describe the notions and methods of Newton-Okounkov bodies and recall some important results from \cite{KAVEH_KHOVANSKII}.
	Let $d \ge 1$.
	Suppose that $S \subset \ZZ^{d+1}$ is a semigroup in $\ZZ^{d+1}$.
	Let $\pi  : \RR^{d+1} \rightarrow \RR$ be the projection into the last component.
	Let $L = L(S)$ be the linear subspace of $R^{d+1}$ which is generated by $S$.
	Let $M = M(S)$ be the rational half-space $M(S):= L(S) \cap \pi^{-1}(\RR_{\ge 0}) = L(S) \cap \left(\RR^d \times \RR_{\ge 0}\right)$, and let $\partial M_\ZZ = \partial M \cap \ZZ^{d+1}$.
	Let $\Con(S) \subset L(S)$ be the closed convex cone which is the closure of the set of all linear combinations $\sum_i \lambda_is_i$ with $s_i \in S$ and $\lambda_i \ge 0$.
	Let $G(S) \subset L(S)$ be the group generated by $S$.
	
	We say that the semigroup $S$ is \emph{non-negative} if $S \subset M$;
	additionally, if $\Con(S)$ is strictly convex and intersects the space $\partial M$ only at the origin, then $S$ is \emph{strongly non-negative} (see \cite[Definition 1.9 and \S 1.4]{KAVEH_KHOVANSKII}).
	
	Following \cite{KAVEH_KHOVANSKII}, when $S$ is non-negative we fix the following notation:
	\begin{itemize}[--]
		\item $[S]_n := S \cap \pi^{-1}(n) = S \cap \left(\ZZ^d \times \{n\}\right)$.
		\item $m(S) := \left[\ZZ : \pi(G(S))\right]$.
		\item $\ind(S) := \left[\partial M_\ZZ : G(S) \cap \partial M\right]$.
		\item $\Delta(S) := \Con(S) \cap \pi^{-1}(m(S))$ \quad  (the Newton-Okounkov body of $S$).
		\item $\Vol_q(\Delta(S))$ is the \emph{integral volume} of $\Delta(M,S)$ (see \cite[Definition 1.13]{KAVEH_KHOVANSKII}); this volume is computed using the
		translation of the \emph{integral measure} on $\partial M$.
	\end{itemize}
	
	The following result is of remarkable importance for us.
	
	\begin{theorem}[{Kaveh-Khovanskii, \cite[Corollary 1.16]{KAVEH_KHOVANSKII}}]\label{thm_limit_KK}
		Suppose that $S$ is strongly non-negative.
		Let $m = m(S)$ and $q = \dim_\RR(\Delta(S))$. 
		Then 
		$$
		\lim_{n\to \infty} \frac{\#[S]_{nm}}{n^q} \;=\; \frac{\Vol_q(\Delta(S))}{\ind(S)}.
		$$
	\end{theorem}
	
	For a subset $U \subset S$ of $S$ and  $n\in \NN$,  we define
	$
	n\star U \;:=\; \Big\{\sum_{i=1}^n u_i \,\mid \, u_1,\ldots, u_n \in U \Big\}.
	$
	For $p \in \NN$, we denote by $\widehat{S}_p$ the subsemigroup of $S$ generated by $[S]_p$, that is,
	$$
	\left[\widehat{S}_p\right]_{np} = n \star [S]_p \quad\; \text{ and } \quad\; \left[\widehat{S}_p\right]_{n} = 0 \text{ if } p \text{ does not divide } n.
	$$
	The following approximation theorem relates the semigroup $S$ with the semigroups $\widehat{S}_{pm}$ for $p$ big enough.
	
	\begin{theorem}[{\cite[Theorem 1.27]{KAVEH_KHOVANSKII}, \cite[Proposition 3.1]{lazarsfeld09}}]
		\label{thm_approx}
		Suppose that $S$ is strongly non-negative.
		Let $m = m(S)$ and $q = \dim(\Delta(S))$. 
		Let $\varepsilon > 0$ be a positive real number.
		Then, for $p \gg 0$  the following statements hold: 
		\begin{enumerate}[(i)]
			\item $\dim_\RR(\Delta(\widehat{S}_{pm})) = q$.
			\item $\ind(\widehat{S}_{pm}) = \ind(S)$. 
			\item We have the inequalities
			$$
			\lim_{n\to \infty} \frac{\#[S]_{nm}}{n^q} - \varepsilon \;\; \le \;\; 
			\lim_{n\to \infty} \frac{\#\left(n \star [S]_{pm}\right)}{n^qp^q} 
			\;\; \le \;\; \lim_{n\to \infty} \frac{\#[S]_{nm}}{n^q}.
			$$
		\end{enumerate} 
	\end{theorem}

	\smallskip
	
	We now briefly recall Minkowski's theorem and the notion of mixed volume of convex bodies.
	Let $\bK = (K_1,\ldots,K_s)$ be a sequence of convex bodies in $\RR^d$.
	For any sequence $\lambda = (\lambda_1,\ldots,\lambda_s) \in \NN^s$ of non-negative integers, we denote by $\lambda \bK$ the Minkowski sum $\lambda \bK := \lambda_1 K_1+\cdots+\lambda_s K_s$ and by $\bK_\lambda$ the multiset $\bK_\lambda := \bigcup_{i=1}^s\bigcup_{j=1}^{\lambda_i} \{K_i\}$ of $\lambda_i$ copies of $K_i$ for each $1 \le i \le s$.
	Below is the classical Minkowski's theorem (see, e.g.,~\cite[Theorem 5.1.7]{SCHNEIDER}).
	
	\begin{theorem}[Minkowski]
		\label{thm_Minkowski}
		 $\Vol_{d}(\lambda \bK)$ is a homogeneous polynomial of degree $d$.	
	\end{theorem}
	 We write the polynomial $\Vol_{d}(\lambda \bK)$ as
	$$
	\Vol_{d}(\lambda \bK) = \sum_{{\dd \in \NN^{s}\; \lvert \dd \rvert=d}}\; \frac{1}{\dd!}\,\MV_d(\bK_\dd) \, \lambda^\dd,
	$$
	where $\MV_d(-)$ denotes the \emph{mixed volume}.
	
	\smallskip
	Next, we concentrate on describing the dimension of a subalgebra of an algebra finitely generated over a field.
	Following the terminology of \cite{KEMPER_TRUNG}, we say that these algebras are \emph{subfinite}.
	Let $\kk$ be a field and $A$ be a $\kk$-algebra. 
	We shall always denote by $\dim(A)$ the Krull dimension of $A$.
	We say that $A$ is \emph{subfinite over $\kk$} if there exists a finitely generated $\kk$-algebra $B$ containing $A$.
	Denote by $\delta_\kk(A)$ the maximal number of elements in $A$ which are algebraically independent over $\kk$, that is, 
	$$
	\delta_\kk(A)  := \sup\big\{ d \mid \text{there exists } x_1,\ldots,x_d \in A \text{ which are algebraically independent over } \kk \big\}.
	$$
	Equivalently, $\delta_\kk(A) = d$ if $d$ is the largest integer such that there is an inclusion $\kk[x_1,\ldots,x_d] \hookrightarrow A$ where $\kk[x_1,\ldots,x_d]$ is a polynomial ring over $\kk$.
	
	\begin{remark}
		If $A$ is finitely generated over $\kk$, then $\dim(A) = \delta_\kk(A)$.
		This is a classical result that follows for instance from the Noether normalization theorem (see, e.g.,~\cite[Theorem 13.3]{EISEN_COMM}).
	\end{remark}
	
	The following theorem  extends the above remark to algebras that are subfinite over $\kk$ (cf. \cite{GIRAL}).
	
	\begin{theorem}[{\cite[Theorem 4.6, Corollary 4.7]{KEMPER_TRUNG}}]
		\label{thm_dim_subfinite_alg}
		Let $A$ be a subfinite algebra over $\kk$.
		Then, the following statements hold:
		\begin{enumerate}[(i)]
			\item $\dim(A) = \delta_\kk(A)$.
			\item If $B \supset A$ is a subfinite algebra over $\kk$, then $\dim(B) \ge \dim(A)$.
		\end{enumerate}	
	\end{theorem}

	\section{Singly graded algebras of almost integral type}
	\label{sect_single_grad}
	
	In this section, we extend the results of \cite[Part II]{KAVEH_KHOVANSKII} regarding the asymptotic behavior of algebras of \emph{almost integral type}.
	Our results are slightly more complete than the ones in \cite[Part II]{KAVEH_KHOVANSKII} with respect to the following points: 
	\begin{enumerate}
		\item We consider an arbitrary field $\kk$ instead of assuming that $\kk$ is algebraically closed.
		\item We substitute the field $\FF$ by any reduced $\kk$-algebra  $R$.
		\item We show that the asymptotic growth of an algebra of almost integral type is determined by its Krull dimension. 
	\end{enumerate}
	Similar results to the ones in this section were also obtained in \cite{cutkosky2014} for graded linear series.

	Let $\kk$ be an arbitrary field and $R$ be a reduced $\kk$-algebra.
	Let $R[t]$ be a standard graded polynomial ring over $R$. 
	Below we introduce the notion of almost integral type in our current setting. 
	
	\begin{definition}
		\begin{enumerate}[(i)]
			\item A graded $\kk$-algebra $A \subset R[t]$ is called of \emph{integral type} if $A$ is finitely generated over $\kk$ and is a finitely generated module over the subalgebra generated by $[A]_1$.
			\item A graded $\kk$-algebra $A \subset R[t]$ is called of \emph{almost integral type} if $A \subset B \subset R[t]$, where $B$ is a graded algebra of integral type.
		\end{enumerate}
	\end{definition}

	Notice that, if $A \subset R[t]$ is a graded $\kk$-algebra of almost integral,  then $\dim_{\kk}\left([A]_n\right) < \infty$ for all $n \in \ZZ$.
	By definition, one has that a graded $\kk$-algebra of almost integral type is a subfinite algebra over $\kk$.
	For any positively graded $\kk$-algebra $A$ which is subfinite over $\kk$, we denote by $G^i(A)$ the finitely generated graded $\kk$-algebra that is generated by the graded components $[A]_0,\ldots,[A]_i$, that is, 
	$$
	G^i(A) \,:=\, \kk\big[[A]_0,\ldots,[A]_i\big] \, \subset \, A.
	$$ 
	
	First, we discuss some properties of graded $\kk$-algebras that are subfinite over $\kk$.
	
	\begin{lemma}
		\label{lem_basic_prop}
		Let $A$ be a positively graded $\kk$-algebra which is subfinite over $\kk$.
		Then, the following statements hold:
		\begin{enumerate}[(i)]
			\item There exists $i_0 \in \NN$ such that $\dim(A) = \dim\left(G^i(A)\right)$ for all $i \ge i_0$.
			\item Suppose that $d = \dim(A) > 0$. Then there exists $h \in \ZZ_+$ and $\alpha \in \RR_{+}$ such that 
			$$
			\dim_{\kk}\big([A]_{nh}\big) > \alpha n^{d-1}
			$$
			for all $n > 0$.
		\end{enumerate}
	\end{lemma}
	\begin{proof}
		(i) This part follows from \autoref{thm_dim_subfinite_alg}.
		
		(ii) Let $C = G^i(A)$ such that $\dim(C) = d$.
		Since $C$ is finitely generated over $\kk$, by \cite[Lemma 13.10, Remark 13.11]{GORTZ_WEDHORN}, one has a positive integer $h > 0$ such that the Veronese subalgebra $C^{(h)} = \bigoplus_{n =0}^\infty [C]_{nh}$ is a standard graded $\kk$-algebra.
		The ring $[A]_0$ is Artinian because by assumption it is a finite dimensional vector space over $\kk$.
		As $\dim(C^{(h)}) = \dim(C) = d$,  by \cite[Corollary of Theorem 13.2]{MATSUMURA} the Hilbert polynomial of $C^{(h)}$ has degree $d-1$ and it coincides with $\text{length}_{[A]_0}\left([C]_{nh}\right)$ for $n \gg 0$. 
		Therefore, there exists a positive real number $\alpha$ such that $\dim_{\kk}([A]_{nh}) \ge \dim_{\kk}([C]_{nh})  > \alpha n^{d-1}$ for all $n > 0$.
	\end{proof}
	
	\begin{lemma}
		\label{lem_dim_short_exact_seq}
		Let $A$, $B$ and $C$ be graded $\kk$-algebras which are subfinite over $\kk$.
		Suppose there exists $n_0 \ge 0$ such that for all $n \ge n_0$ we have a short exact sequence $0 \rightarrow [A]_n \rightarrow [B]_n \rightarrow [C]_n \rightarrow 0$.
		Then $\dim(B) = \max\{\dim(A), \dim(C)\}$.
	\end{lemma}
	\begin{proof}
		After choosing some $h \ge n_0$, we  substitute $A$, $B$ and $C$ by the Veronese subalgebras $A^{(h)}$, $B^{(h)}$ and $C^{(h)}$, respectively.
		Thus we may assume that $0 \rightarrow [A]_n \rightarrow [B]_n \rightarrow [C]_n \rightarrow 0$ is exact for all $n \ge 1$.
		
		By invoking \autoref{thm_dim_subfinite_alg}, we choose $i > 0$ such that $\dim(G^i(B) \cap A) = \dim(A)$, $\dim(G^i(B)) = \dim(B)$ and $\dim(G^i(C)) = \dim(C)$.
		Notice that we have the short exact sequence 
		$$
		0 \rightarrow \left[G^i(B) \cap A\right]_+ \rightarrow \left[G^i(B)\right]_+ \rightarrow \left[G^i(C)\right]_+ \rightarrow 0,
		$$
		where $[S]_+ = \bigoplus_{n=0}^\infty [S]_n$ for any graded $\kk$-algebra $S$.
		From the above short exact sequence, we obtain that the ideal $\left[G^i(B) \cap A\right]_+$ is finitely generated (as it is an ideal over $G^i(B)$), which implies that $G^i(B) \cap A$ is a finitely generated graded $\kk$-algebra (see, e.g.,~\cite[Proposition 1.5.4]{BRUNS_HERZOG}).
		We substitute $A$, $B$ and $C$ by $G^i(B) \cap A$, $G^i(B)$ and $G^i(C)$, respectively. 
		So, we assume that $A$, $B$ and $C$ are finitely generated $\kk$-algebras.
		
		Finally, by standard arguments we can show that $\dim(B) = \max\{ \dim(A), \dim(C)\}$.
		We choose $h > 0$ such that $0 \rightarrow A^{(h)} \rightarrow B^{(h)} \rightarrow C^{(h)} \rightarrow 0$ is a short exact sequence of standard graded $\kk$-algebras (see, e.g.,~\cite[Lemma 13.10, Remark 13.11]{GORTZ_WEDHORN}), and then the result follows from the additivity of Hilbert polynomials and \cite[Corollary of Theorem 13.2]{MATSUMURA}.
	\end{proof}
	
	The following easy remark shows that we can always substitute $R$ by a reduced finitely generated $\kk$-algebra (cf. \cite[Proposition 2.25]{KAVEH_KHOVANSKII}).
	
	\begin{remark}
		\label{rem_finite_gen}
		Let $A \subset R[t]$ be a graded $\kk$-algebra of almost integral type. 
		By definition, let $B \supset A$ be a graded $\kk$-algebra of integral type and suppose that $f_{i_1}t^{i_1}, \ldots, f_{i_m}t^{i_m}$ are generators of $B$ as a $\kk$-algebra.
		Let $R' \subset R$ be the reduced $\kk$-algebra generated by the elements $f_{i_1},\ldots,f_{i_m} \in R$.
		So, one has that $A \subset R'[t]$ with $R'$ being a reduced finitely generated $\kk$-algebra. 
	\end{remark}
	
	Thus, as a consequence of the above remark, we now safely assume that $R$ is a reduced finitely generated algebra.
	Hence, let $\pp_1,\ldots,\pp_\ell \in \Spec(R)$ be the minimal primes of $R$.
	Since $R$ is assumed to be reduced, we have a canonical inclusion 
	$
	R[t] \, \hookrightarrow \, \bigoplus_{i=1}^\ell R/\pp_i[t]
	$. 
	Let $A \subset R[t]$ be a graded $\kk$-algebra of almost integral type.
	For each $1 \le i \le \ell$, let $B^i$ be the $\kk$-subalgebra defined by 
	\begin{equation}
		\label{eq_alg_B^i}
		\left[B^i\right]_n \,:= \, \begin{cases}
			\kk     \qquad\quad \;\, \text{ if } n = 0 \\
			\left[M^i\right]_n \quad\, \text{ if } n > 0
		\end{cases}
	\end{equation}
	where
	$M^i := \left(A \cap \pp_1R[t] \cap \cdots \cap \pp_{i-1}R[t]\right) / \left(A \cap \pp_1R[t] \cap \cdots \cap \pp_{i-1}R[t] \cap \pp_{i}R[t]\right)$.
	By construction, $B^i$ is a graded  $\kk$-algebra of almost integral type and we have a canonical inclusion $B^i \hookrightarrow R/\pp_i[t]$.
	
	\begin{lemma}
		\label{lem_decompose}
		Under the above notation, the following statements hold:
		\begin{enumerate}[(i)]
			\item $\dim\left(A\right) = \max\big\{\dim\left(B^i\right) \mid 1 \le i \le \ell \big\}$.
			\item For all $n > 0$, one has $\dim_{\kk}\left([A]_n\right) = \sum_{i=1}^{\ell} \dim_{\kk}\left(\left[B^i\right]_n\right)$.
		\end{enumerate}
	\end{lemma}
	\begin{proof}
		For each $1 \le i \le \ell$ we have the short exact sequence
		$$
		0 \,\rightarrow\, M^i \,\rightarrow\, \frac{A}{A \cap \pp_1R[t] \cap \cdots \cap \pp_{i-1}R[t] \cap \pp_{i}R[t]} \,\rightarrow\, \frac{A}{A \cap \pp_1R[t] \cap \cdots \cap \pp_{i-1}R[t]} \rightarrow 0.
		$$
		Since $R$ is reduced, $\pp_1 R[t] \cap \cdots\cap \pp_\ell R[t] = 0$.
		Therefore, part (ii) is clear and part (i) follows from \autoref{lem_dim_short_exact_seq}.
	\end{proof}
	
	The following theorem contains the main result of this section. 
	By using the general arguments above, the first idea in the proof is to reduce to the case when $R$ is a domain. 
	For organizational purposes, in the next subsection, we encapsulate a proof of \autoref{thm_limit_graded} under the assumptions of $R$ being a finitely generated $\kk$-domain (see \autoref{thm_limit_graded_domain}).
	
	\begin{theorem}
		\label{thm_limit_graded}
		Let $\kk$ be a field and $R$ be a reduced $\kk$-algebra.
		Let $A \subset R[t]$ be a graded $\kk$-algebra of almost integral type. 
		Suppose $d = \dim(A) > 0$.
		Then, there exists an integer $m > 0$ such that the limit 
		$$
		\lim_{n\to \infty} \frac{\dim_{\kk}\left([A]_{nm}\right)}{n^{d-1}}  \;\in\; \RR_{+}
		$$
		exists and it is a positive real number.
	\end{theorem}
	\begin{proof}
		By \autoref{rem_finite_gen}, we assume that $R$ is finitely generated over $\kk$.
		Suppose that $\pp_1,\ldots,\pp_\ell$ are the minimal primes of $R$, and consider the algebras $B^i$ constructed in \autoref{eq_alg_B^i}.
		Since each $B^i \subset R/\pp_i[t]$ is an algebra of almost integral type, \autoref{thm_limit_graded_domain} yields the existence of the following limits
		$$
		\lim_{n\to \infty} \frac{\dim_{\kk}\left(\left[B^i\right]_{nm_i}\right)}{n^{d_i-1}}  \;\in\; \RR_{+}
		$$
		where $m_i = m(B^i)$ and $d_i = \dim(B^i)$.
		Therefore, after taking $m = \text{lcm}(m_1,\ldots,m_\ell)$, the result of the theorem follows from \autoref{lem_decompose}.
	\end{proof}

	\subsection{The case where $R$ is a domain}
	\label{subsect_R_domain}
	The setup below is used throughout this subsection.
	
	\begin{setup}
		\label{setup_domain_grad}
		Let $\kk$ be a field, $R$ be a finitely generated $\kk$-domain and $r = \trdeg_\kk(\Quot(R))$.
		Let $A \subset R[t]$ be a graded $\kk$-algebra of almost integral type.
		Let $d = \dim(A)$.
		Let $m(A)$ be the index $m(A) := [\ZZ:G]$ of the subgroup $G$  of $\ZZ$ generated by $\{n \in \NN \mid [A]_n \neq 0\}$.
		For any $p > 0$, we denote by $\widehat{A}_p$ the graded $\kk$-subalgebra $\widehat{A}_p := \kk\big[[A]_p\big] \subset A$ generated by the graded part $[A]_p$.
		By regrading the algebra $\widehat{A}_p$, we obtain the standard graded $\kk$-algebra $\widetilde{A}_p$ defined by ${[\widetilde{A}_p]}_{n} := {[\widehat{A}_p]}_{np}$ for all $n \ge 0$.
		As customary, let 
		$$
		e(\widetilde{A}_p) \;:=\;  (\dim(\widetilde{A}_p)-1)! \;\cdot\, \lim_{n\to \infty} \frac{\dim_{\kk}\big({[\widetilde{A}_p]}_{n}\big)}{n^{\dim(\widetilde{A}_p)-1}} 
		$$ be the multiplicity of the standard graded $\kk$-algebra $\widetilde{A}_p$.
	\end{setup}

	\begin{remark}
		\label{rem_order_ZZ_r}
		We fix an order on $\ZZ^r$ which is compatible with an addition, i.e.,~for any $\bn,\bm,\mathbf{p}\in \ZZ^r$ with $\bn<\bm$, we have $\bn+\mathbf{p}<\bm+\mathbf{p}$. 
		One way to do this is to fix a linear function $l$ on $\ZZ^r$ with rationally independent coefficients and set $\bn<\bm \iff l(\bn)<l(\bm)$ for $\bn,\bm\in \ZZ^r$.
	\end{remark}

	We assume that $\Quot(R)$ admits a valuation $\nu: \Quot(R) \to \ZZ^r$ such that $\nu (\alpha) = 0$ for all $\alpha \in \kk \subset \Quot(R)$.
	We further suppose that $\nu: \Quot(R) \to \ZZ^r$ is \emph{faithful} and has \emph{leaves of bounded dimension}.
	By \autoref{prop_bounded_val} we can always construct such a valuation.
	The faithfulness of $\nu$ means that $\nu(\Quot(R)) = \ZZ^r$.
	For any $\bn = (n_1,\ldots,n_r) \in \ZZ^r$, we define 
	$
	K_\bn :=\{f\in \Quot(R)\mid \nu(f)\gs \bn\}
	$ 
	and 
	$K_\bn^+ :=\{f\in \Quot(R)\mid \nu(f)> \bn\}
	$,
	and we say that the \emph{leaf with value $\bn$} is given by 
	$$
	L_\bn \; := \; K_\bn / K_\bn^+.
	$$
	We say that $\nu$ has leaves with bounded dimension if $\sup_{\bn \in \ZZ^r} \left(\dim_{\kk}\left(L_\bn\right)\right) < \infty$.
	Let $\ell$ be the maximal dimension of the leaves of $\nu$:
	$$
	\ell \;:=\; \max_{\bn\in \ZZ^r}\big(\dim_{\kk}\left( L_\bn \right)\big).
	$$
	For every $t \ge 1$, we define 
	\begin{equation}
		\label{eq_semigroups}
		\Gamma^{(t)}_A \, := \; \Big\{  (\bn,n) = (n_1,\ldots,n_r,n) \in \ZZ^{r} \times \NN \;\mid\; \dim_{\kk}\left( K_{\bn} \cap [A]_n/ K_{\bn}^+ \cap [A]_n\right) \ge t \Big\}.
	\end{equation}
	For any $n \ge 0$ and $\bn \in \ZZ^r$ one has the inequalities
	\begin{equation*}
		\dim_\kk\left( K_\bn \cap [A]_n/ K_\bn^+ \cap [A]_n\right) \; \ls \; \dim_\kk\left( L_\bn \right) \; \ls \; \ell.
	\end{equation*}
	We consider the integer
	\begin{equation}
		\label{eq_ell_A}
		\ell_A \; := \;  \max\big\lbrace t \in \NN \mid \Gamma^{(t)}_A \not\subseteq \{0\} \big\rbrace \; \le \; \ell.
	\end{equation}
	Then, we have the following equalities 
	\begin{align}
		\label{eq_equal_semigroup}
		\begin{split}
			\dim_{\kk}\left([A]_n\right) &= \dim_{\kk}\left(\bigoplus_{\bn \in \ZZ^r} \left( K_\bn \cap [A]_n/ K_\bn^+ \cap [A]_n\right)\right) \\
			&= \sum_{t=1}^{\ell_A} \#\left[\Gamma^{(t)}_A\right]_n
		\end{split}
	\end{align}
	for all $n \ge 0$.

	We have the following general lemma.
	
	\begin{lemma}
		\label{lem_polytope_equality}
		Let  $S\subset \ZZ^{r} \times \NN$ be a strongly non-negative semigroup and let $\Gamma\subset S$ be a non-trivial semigroup ideal (i.e.,~$a + S \subset \Gamma$ for all $a \in S$), then 
		\begin{enumerate}[(i)]
			\item $\Delta(\Gamma) = \Delta(S)$.
			\item $\ind(\Gamma) = \ind(S)$.
			\item $m(\Gamma) = m(S)$
		\end{enumerate}
	\end{lemma}
	\begin{proof}
		Let $a\in \Gamma$ be any point, then we have an inclusion $a+S\subset \Gamma$, and hence $G(S)=G(\Gamma)$ and, in particular, $\ind(\Gamma) = \ind(S)$ and $m(\Gamma) = m(S)$.
		Moreover, we obtain inclusions of closed convex cones
		$$
		\Con(S) + a \subset \Con(\Gamma) \subset \Con(S),
		$$
		which guaranties the equalities $\Con(S)= \Con(\Gamma)$ and $\Delta(\Gamma) = \Delta(S)$.
	\end{proof}

	In the next proposition, we study some properties of $\Gamma^{(t)}_A$.
	For the special case $t = 1$, we simply write $\Gamma_A := \Gamma^{(1)}_A$.
	
	\begin{proposition}
		\label{prop_properties_semig}
		Suppose that $1 \le t \le \ell_A$ (that is,  $\Gamma^{(t)}_A \not\subseteq \{0\}$).
		Then, the following statements hold: 
		\begin{enumerate}[(i)]
			\item $\Gamma^{(t)}_A \subseteq \Gamma_A$.
			\item $\Gamma_A$ is a semigroup, and $\Gamma^{(t)}_A$ is a semigroup ideal of $\Gamma_A$.
			In particular, one has that \[\Delta(\Gamma^{(t)}_A) = \Delta(\Gamma_A),\quad  \ind(\Gamma^{(t)}_A) = \ind(\Gamma_A)\quad \text{and}\quad m(\Gamma^{(t)}_A) = m(A).
			\]
			\item $\Gamma^{(t)}_A$ is strongly non-negative.
		\end{enumerate}
	\end{proposition}
	\begin{proof}
		(i) This part is clear.
		
		(ii) Let $(\bn,n) = (n_1,\ldots,n_r,n) \in \Gamma^{(t)}_A$ and $(\bm,m) = (m_1,\ldots,m_r,m) \in \Gamma_A$.
		Then we can choose elements $f_1,\ldots,f_t \in [A]_n$ such that $\nu(f_i)=\bn$ and the classes $\overline{f_1},\ldots,\overline{f_t}$ are linearly independent over $\kk$ in $K_\bn \cap [A]_n/ K_\bn^+ \cap [A]_n$.
		Similarly, let $g \in [A]_m$ such that $\nu(g)=\bm$.
		
		By contradiction, suppose that the classes of $gf_1,\ldots,gf_t \in [A]_{n+m}$ are $\kk$-linearly dependent in 
		$$
		K_{\bn+\bm} \cap [A]_{n+m}/ K_{\bn+\bm}^+ \cap [A]_{n+m}.
		$$
		Thus, there exist $a_1\ldots,a_t \in \kk$ such that $\nu(a_1gf_1+\cdots+a_tgf_t) > \bn + \bm$, which implies that $\nu(a_1f_1+\cdots+a_tf_t) > \bn$.
		But, the last inequality $\nu(a_1f_1+\cdots+a_tf_t) > \bn$ yields the contradiction $a_1\overline{f_1}+\cdots+a_t\overline{f_t} = 0 \in K_\bn \cap [A]_n/ K_\bn^+ \cap [A]_n$.
		So, as required we have $(\bn+\bm,n+m) \in \Gamma^{(t)}_A$.
		
		Finally, \autoref{lem_polytope_equality} implies that $\Delta(\Gamma^{(t)}_A) = \Delta(\Gamma_A)$,  $\ind(\Gamma^{(t)}_A) = \ind(\Gamma_A)$ and $m(\Gamma^{(t)}_A) = m(\Gamma_A)$.
		Also, notice that $m(\Gamma_A) = m(A)$.
		
		(iii)  Since $\Gamma^{(t)}_A \subset \Gamma_A$, it is enough to show that 
		$\Gamma_A$
		is strongly non-negative.
		By definition, there is a graded $\kk$-algebra $B \subset R[t]$ of integral type such that $A \subset B$.
		Consider the semigroup 
		$$
		\Gamma_B := \big\{  (\bn, n) \in \ZZ^{r} \times \NN \;\mid\;  K_{\bn} \cap [B]_n/ K_{\bn}^+ \cap [B]_n \neq 0 \big\}
		$$
		determined by $B$. 
		After possibly extending $B$ by a bigger algebra of integral type, we can assume that $G(\Gamma_B) = \ZZ^{r+1}$ (see \cite[Lemma 2.29]{KAVEH_KHOVANSKII}).
		Since $B$ is a finitely generated module over a standard graded $\kk$-algebra and $\dim(B) \le r+1$, it follows that $\dim_{\kk}([B]_n)$ becomes a polynomial of degree bounded by $r$ for $n \gg0$.
		So, \cite[Theorem 1.18]{KAVEH_KHOVANSKII} implies that $\Gamma_B$ is strongly non-negative, and since $\Gamma_A \subset \Gamma_B$, the result follows.
	\end{proof}
	
	The following theorem deals with the asymptotic behavior of the growth of $A$.
	The semigroup $\Gamma_A$ is given by
	$$
	\Gamma_A \;=\; \Gamma^{(1)}_A \; = \; \big\lbrace (\nu(a), n) \in \ZZ^{r} \times \NN \;\mid\;  0 \neq a \in [A]_n \big\rbrace
	$$ 
	and we call it the \emph{valued semigroup} of the graded $\kk$-algebra $A$.
	To simplify notation, we denote $\Delta(\Gamma_A)$ and $\ind(\Gamma_A)$ by $\Delta(A)$ and $\ind(A)$, respectively.
	
	\begin{theorem}
		\label{thm_limit_graded_domain}
		Assume \autoref{setup_domain_grad}, set $m = m(A)$ and let $d = \dim(A)$ and $\ell_A$ be as in \autoref{eq_ell_A}.
		Then
		$$
		\lim_{n\to \infty} \frac{\dim_{\kk}\left([A]_{nm}\right)}{n^{d-1}}  \; = \; \ell_A \cdot \frac{\Vol_{d-1}(\Delta(A))}{\ind(A)},
		$$
		and we have the equality 
		$$
		(d-1)! \, \cdot \, \lim_{n\to \infty} \frac{\dim_{\kk}\left([A]_{nm}\right)}{n^{d-1}} \;\;= \;\; \lim_{p \to \infty}\, \frac{e(\widetilde{A}_{pm})}{p^{d-1}}.
		$$
	\end{theorem}
	\begin{proof}
		For any $p > 0$ and  $t>0$ we consider the semigroup
		\begin{equation*}
			\label{eq_approx_semigroup}
			\Gamma_{\widehat{A}_{pm}}^{(t)} = \Big\{  (\bn, n) \in \ZZ^{r} \times \NN \;\mid\;  \left(K_{\bn} \cap \big[\widehat{A}_{pm}\big]_n/ K_{\bn}^+ \cap \big[\widehat{A}_{pm}\big]_n\right) \ge t \Big\};
		\end{equation*}
		these semigroups play the same role for $\widehat{A}_{pm}$ than the semigroups in \autoref{eq_semigroups} for $A$.
		
		By using the inclusions 
		\begin{equation}
			\label{eq_inclusion_semig}
			n \star \left[\Gamma^{(t)}_A\right]_{pm} \; \subset \; \left[\Gamma_{\widehat{A}_{pm}}^{(t)}\right]_{npm} \; \subset \; \left[\Gamma^{(t)}_A\right]_{npm}
		\end{equation}
		and \autoref{thm_approx}, we can choose $p \gg 0$ such that 
		$$
		\dim_\RR\big(\Delta\big(\Gamma_{\widehat{A}_{pm}}^{(t)}\big)\big) = \dim_\RR\big(\Delta\big(\Gamma^{(t)}_A\big)\big).
		$$
		From \autoref{thm_dim_subfinite_alg}, by possibly making $p$ larger, we can assume that $\dim(\widehat{A}_{pm}) = d$.	
		After regrading $\widehat{A}_{pm}$ and considering the standard graded $\kk$-algebra $\widetilde{A}_{pm}$, we obtain the existence of a polynomial $Q_{pm}(n)$ of degree $d-1$ such that $\dim_{\kk}([\widehat{A}_{pm}]_{npm}) = Q_{pm}(n)$ for $n \gg 0$ (see, e.g.,~\cite[Theorem 4.1.3]{BRUNS_HERZOG}).
		Similarly to \autoref{eq_equal_semigroup}, we have the equality 
		\begin{equation}
			\label{eq_sum_A_p}
			\dim_{\kk}\left(\big[\widehat{A}_{pm}\big]_{npm}\right) = \sum_{t=1}^{\ell_{\widehat{A}_{pm}}} \#\left[\Gamma_{\widehat{A}_{pm}}^{(t)}\right]_{npm}.
		\end{equation}
		From the inclusion $\Gamma_{\widehat{A}_{pm}}^{(t)} \subset \Gamma_{\widehat{A}_{pm}}$, it then necessarily follows that the growth of $\#\big[\Gamma_{\widehat{A}_{pm}}\big]_{npm}$ is asymptotically a polynomial of degree $d-1$.
		So, by applying \autoref{thm_limit_KK} to the semigroup $\Gamma_{\widehat{A}_{pm}}$, we obtain 
		$$
		\dim_\RR\big(\Delta\big(A\big)\big) = \dim_\RR\big(\Delta\big(\Gamma_{\widehat{A}_{pm}}\big)\big) = d-1.
		$$ 
		From \autoref{prop_properties_semig}, we also have $\Delta(\Gamma^{(t)}_A) = \Delta(A)$,  $\ind(\Gamma^{(t)}_A) = \ind(A)$ and $m(\Gamma^{(t)}_A) = m(A) = m$ for all $1 \le t \le \ell_A$.
		Therefore, by using \autoref{eq_equal_semigroup}, \autoref{prop_properties_semig} and \autoref{thm_limit_KK}, we obtain the equality
		$$
		\lim_{n\to \infty} \frac{\dim_{\kk}\left([A]_{nm}\right)}{n^{d-1}}  \; = \; \ell_A \cdot \frac{\Vol_{d-1}(\Delta(A))}{\ind(A)}.
		$$
		So, the first statement of the theorem holds.
		
		From \autoref{eq_inclusion_semig} and \autoref{thm_approx}, for any $\varepsilon > 0$, we can choose $p \gg 0$ such that 
		\begin{equation*}
			\lim_{n\to \infty} \frac{\#\big[\Gamma^{(t)}_A\big]_{nm}}{n^{d-1}} - \varepsilon \;\; \le \;\; 
			\lim_{n\to \infty} \frac{\#\Big[\Gamma_{\widehat{A}_{pm}}^{(t)}\Big]_{npm}}{n^{d-1}p^{d-1}} 
			\;\; \le \;\; \lim_{n\to \infty} \frac{\#\big[\Gamma^{(t)}_A\big]_{nm}}{n^{d-1}}
		\end{equation*}
		for all $t$.
		Finally, by \autoref{eq_sum_A_p} we obtain 
		$$
		\lim_{n\to \infty} \frac{\dim_{\kk}\left([A]_{nm}\right)}{n^{d-1}} 
		\; = \; 
		\lim_{p \to \infty} \frac{1}{p^{d-1}} \left(\lim_{n\to \infty} \frac{\dim_\kk\big(\big[\widehat{A}_{pm}\big]_{npm}\big)}{n^{d-1}} \right) \; = \; \frac{1}{(d-1)!} \lim_{p \to \infty} \frac{e(\widetilde{A}_{pm})}{p^{d-1}},
		$$
		which gives the second claimed equality.
	\end{proof}

	\subsection{Valuations with bounded leaves}
	\label{subsect_bounded}
	In this subsection, we show that the valuations used in \autoref{subsect_R_domain} can always be constructed.
	Here we continue using \autoref{setup_domain_grad}.
	We start with the following lemma that will allow us construct a valuation on $\Quot(R)$.
	
	\begin{lemma}
		\label{lem_exists_reg_ring}
		There exists a regular local ring $(S,\mm,\kk')$ such that $R \subset S$, $\Quot(S) = \Quot(R)$ and the residue field $\kk' = S/\mm$ is finite over $\kk$.
	\end{lemma}
	\begin{proof}
		From \cite[\href{https://stacks.math.columbia.edu/tag/07PJ}{Proposition 07PJ}]{stacks-project} and \cite[\href{https://stacks.math.columbia.edu/tag/07P7}{Definition 07P7}]{stacks-project}, we have 
		$$
		\text{Reg}(R) = \big\{ \pp \in \Spec(R) \mid R_\pp \text{ is a regular local ring}\big\}
		$$
		is an open subset of $\Spec(R)$.
		As $R$ is a domain, $\text{Reg}(R)$ is non-empty.
		Then we can  choose $0 \neq g \in R$ such that $\Spec(R_g) \subset \text{Reg}(R)$.
		Let $\mm \subset \Spec(R_g)$ be a maximal ideal of $R_g$.
		Therefore, by setting $S = {(R_g)}_\mm$ the result follows.
	\end{proof}
	
	From \autoref{lem_exists_reg_ring}, fix a regular local ring $(S,\mm,\kk')$ such that $R \subset S$, $\Quot(S) = \Quot(R)$ and the residue field $\kk' = S/\mm$ is finite over $\kk$. 
	Let $r = \dim(S) = \trdeg_\kk\left(\Quot(R)\right)$, and choose $y_1,\ldots,y_r  \subset \mm$ a regular system of parameters for $S$.
	Let $b_1,\ldots, b_r$ be rationally independent  real numbers with $b_i\gs 1$ for every $1\ls i\ls r$.
	Since $S$ is a regular local ring and $\Quot(S) = \Quot(R)$, we can construct a valuation $\omega$ on $\Quot(R)$ with values in $\RR$ by setting 
	$$
	\omega\left(y_1^{n_1}\cdots y_r^{n_r}\right) = n_1b_1 + \cdots + n_rb_r \; \in \; \RR
	$$ 
	for every $(n_1,\ldots,n_d)\in \NN^r$, and $\omega(\gamma) = 0$ if $\gamma \in S$ has non-zero residue in $\kk' = S/\mm$. 
	Let $(V_\omega, \mm_\omega, \kk_\omega)$ be the valuation ring of $\omega$ in $\Quot(R)$.
	Notice that $V_\omega$ dominates $S$ and we have the equality $\kk_\omega = \kk'$.
	Since $b_1,\ldots,b_r$ are rationally independent, we can define a function $\varphi : \Quot(R) \rightarrow \ZZ^r$ determined by $\varphi(f) = (n_1,\ldots,n_r)$ if $f \in \Quot(R)$ and $\omega(f) = n_1b_1+\cdots+n_rb_r$.
	As in \autoref{rem_order_ZZ_r}, we now fix the order on $\ZZ^r$ that is determined by the linear function $l : \ZZ^r \rightarrow \RR, \; (n_1,\ldots,n_r) \mapsto n_1b_1 + \cdots + n_rb_r$.
	Therefore, we obtain the valuation $\nu : \Quot(R) \rightarrow \ZZ^r, \; f \mapsto \varphi(f)$ for which the proposition below is valid.
	
	\begin{proposition}
		\label{prop_bounded_val}
		There is a valuation $\nu : \Quot(R) \rightarrow \ZZ^r$ that satisfies the following conditions:
		\begin{enumerate}[(i)]
			\item  $\nu (\alpha) = 0$ for all $\alpha \in \kk \subset \Quot(R)$.
			\item $\nu$ is faithful.
			\item $\nu$ has leaves of bounded dimension.
		\end{enumerate}
	\end{proposition}

	\section{Multigraded algebras of almost integral type}
	\label{sect_mult_grad}
	
	Here we define and study a volume function for multigraded algebras of almost integral type (see \autoref{eq_function_A}), and we relate this function to certain global Newton-Okounkov bodies.
	In \cite{lazarsfeld09}, similar volume functions have been considered for the case of multigraded linear series in an irreducible projective variety over an algebraically closed field.
	First, for the sake of completeness, we recall the notion of mixed multiplicities for the well-studied case of standard multigraded algebras, and we explain how this volume function encodes the mixed multiplicities.
	
	\begin{remark}
		\label{rem_standard_graded}
		Let $\kk$ be a field and $A$ be a standard $\NN^s$-graded $\kk$-algebra (i.e.,~it is finitely generated over $\kk$ by elements of degree $\ee_i$ for $1 \le i \le s$).
		Suppose $A$ is a domain and set $d = \dim(A)$.
		Let $P_A(\bn)=P_A(n_1,\ldots,n_s)$ be the multigraded Hilbert polynomial of $A$ (see, e.g.,~\cite[Theorem 4.1]{HERMANN_MULTIGRAD}, \cite[Theorem 3.4]{MIXED_MULT}).
		Then, the degree of $P_A$ is equal to $q = d - s$ and 
		$
		P_A(\bn) = \dim_\kk\left([A]_\bn\right) 
		$
		for all $\bn \in \NN^s$ such that $\bn \gg \mathbf{0}$.
		Furthermore, if we write 
		\begin{equation*}
			P_{A}(\bn) = \sum_{d_1,\ldots,d_s \ge 0} e(d_1,\ldots,d_s)\binom{n_1+d_1}{d_1}\cdots \binom{n_s+d_s}{d_s},
		\end{equation*}
		then $0 \le e(d_1,\ldots,d_s) \in \ZZ$ for all $d_1+\cdots+d_s = q$.
		For $\dd = (d_1,\ldots,d_s)$ with $|\dd| = q$, we say that 
		$$
		e(\dd; A) \; := \; e(d_1,\ldots,d_s)
		$$ 
		is the \emph{mixed multiplicity of $A$ of type $\dd$}.
		We define a function 
		$$
		F_A(\bn) \;:=\; \lim_{n \to \infty} \frac{\dim_{\kk}\left([A]_{n\bn}\right)}{n^q}  \;= \; \lim_{n \to \infty} \frac{\dim_{\kk}\left([A]_{(nn_1,\ldots,nn_s)}\right)}{n^q}.
		$$
		For all $\bn  \in \ZZ_+^s$, we have $F_A(\bn) = G_A(\bn)$ where $G_A(\bn)$ is the homogeneous polynomial 
		$$
		G_A(\bn) \;:=\;  \sum_{|\dd| = q} \frac{1}{\dd!} \,e(\dd; A)\, \bn^\dd \;=\;  \sum_{d_1+\cdots+d_s = q} \frac{e(d_1,\ldots,d_s; A)}{d_1!\cdots d_s!} \, n_1^{d_1}\cdots n_s^{d_s}.
		$$
		So, the function $F_A(\bn)$ encodes the mixed multiplicities of $A$. 
	\end{remark}
	
	Let $\kk$ be a field and $R$ be a $\kk$-domain.
	We introduce new variables $t_1,\ldots,t_s$ over $R$ and consider $R[t_1,\ldots,t_s]$ as a standard $\NN^s$-graded polynomial ring where $\deg(t_i) = \ee_i  \in \NN^s$.
	We have the following definition.
	
	\begin{definition}
		\label{def_multigrad_integral}
		\begin{enumerate}[(i)]
			\item An $\NN^s$-graded $\kk$-algebra $A = \bigoplus_{\bn \in \NN^s} [A]_{\bn} \subset R[t_1,\ldots,t_s]$ is called of \emph{integral type} if $A$ is finitely generated over $\kk$ and is a finite module over the subalgebra generated by $[A]_{\ee_1}, [A]_{\ee_2},\ldots,[A]_{\ee_s}$.
			\item An $\NN^s$-graded $\kk$-algebra $A = \bigoplus_{\bn \in \NN^s} [A]_{\bn} \subset R[t_1,\ldots,t_s]$ is called of \emph{almost integral type} if $A \subset B \subset R[t_1,\ldots,t_s]$, where $B$ is an $\NN^s$-graded algebra of integral type.
		\end{enumerate}
	\end{definition}
	
	The following setup is used throughout the rest of this section.
	\begin{setup}
		\label{setup_mult_grad}
		Let $\kk$ be a field and $R$ be a $\kk$-domain.
		Let $A \subset R[t_1,\ldots,t_s]$ be an $\NN^s$-graded $\kk$-algebra of almost integral type.
		Set $d = \dim(A)$ and $q = d-s$.
		We assume that $[A]_{\ee_i} \neq 0$ for all $1 \le i \le s$.
	\end{setup}

	Let $\bn = (n_1,\ldots,n_s)  \in \ZZ_+^s$. 
	We consider the graded $\kk$-algebra 
	\begin{equation}
		\label{eq_veronese_subalg}
		A^{(\bn)} \;:=\; \bigoplus_{n =0}^\infty [A]_{n\bn}.
	\end{equation}
	By definition, $A^{(\bn)}$ is a graded $\kk$-algebra of almost integral type with a canonical inclusion $A^{(\bn)} \hookrightarrow R[t]$ into a standard graded polynomial ring $R[t]$.
	Denote by $A_{(0)} \subset \Quot(A)$ the subfield of fractions of multi-homogeneous elements with the same degree, that is, 
	$$
	A_{(0)} \;:=\; \Big\lbrace \frac{x}{y} \in \Quot(A) \mid x,y \in A \text{ multi-homogeneous elements with } y\neq 0 \text{ and } \deg(x) = \deg(y) \Big\rbrace.
	$$
	Let $0 \neq f_i \in [A]_{\ee_i}$ for each $1 \le i \le s$.
	Since $\Quot(A) = A_{(0)}\left(f_1,\ldots,f_s\right)$ with $f_1,\ldots,f_s$ transcendental elements over $A_{(0)}$, \autoref{thm_dim_subfinite_alg} yields that $\dim(A) = s + \trdeg_\kk\left(A_{(0)}\right)$.
	In the same way, $\Quot(A^{(\bn)}) = A_{(0)} \left(f_1^{n_1}\cdots f_s^{n_s}\right)$ with $f_1^{n_1}\cdots f_s^{n_s}$ transcendental over $A_{(0)}$, and \autoref{thm_dim_subfinite_alg} yields that $\dim(A^{(\bn)}) = 1 + \trdeg_\kk\left(A_{(0)}\right)$.
	It follows that $\dim(A^{(\bn)}) = \dim(A) - s +1 = d-s+1$. 
	Therefore, for $\bn  \in \ZZ_+^s$,  \autoref{thm_limit_graded_domain} applied to the algebra $A^{(\bn)}$ gives a well-defined function 
	\begin{equation}
		\label{eq_function_A}
		F_A(\bn) \;:= \; \lim_{n \to \infty} \frac{\dim_{\kk}\left([A]_{n\bn}\right)}{n^q}
	\end{equation}
	with $q = d -s$.
	We say that $F_A(\bn)$ is the \emph{volume function} of the $\NN^s$-graded $\kk$-algebra $A$. 
	
	The following simple example shows that, for an arbitrary algebra of almost integral type $A$, $F_A(\bn)$ may not coincide with a polynomial (also, see \cite[\S 7]{cutkosky2019}).
	So, in general we do not have a suitable extension of \autoref{rem_standard_graded}.
	
	\begin{example}
	    \label{examp_non_poly}
		Let $R=\kk[x]$ be a polynomial ring and consider $R[t_1,t_2]$.
		Let $\alpha(n_1,n_2) = \Big\lceil 2\sqrt{n_1^2+n_2^2} \Big\rceil$, where $\lceil \beta \rceil$ denotes the ceiling function of a real number $\beta \in \RR$.
		We define the family of vector spaces
		$$
		[A]_{(n_1,n_2)} \;=\; \left( \bigoplus_{j=\alpha(n_1,n_2)}^{2(n_1+n_2)} \, \kk \cdot x^j \right) t_1^{n_1} t_2^{n_2} \;\subset\; {\Big[R[t_1,t_2]\Big]}_{(n_1,n_2)}
		$$
		over $\kk$.
		Then, $A = \bigoplus_{(n_1,n_2) \in \NN^2} [A]_{(n_1,n_2)} \subset R[t_1,t_2]$ is naturally an $\NN^2$-graded $\kk$-algebra.
		Since $A \subset B$ with $B = \kk\left[xt_1, x^2t_1, xt_2, x^2t_2\right]$, it follows that $A$ is an $\NN^2$-graded $\kk$-algebra of almost integral type. 
		However, in this case the corresponding function $F_A(n_1,n_2)$ is equal to 
		\begin{align*}
			F_A(n_1,n_2) &= \lim_{n \to \infty} \frac{\dim_{\kk}\left([A]_{(nn_1,nn_2)}\right)}{n}\\ &= \lim_{n \to \infty} \frac{2(nn_1+nn_2) - \alpha(nn_1,nn_2) +1}{n} \\
			& = 2(n_1+n_2) -  2\sqrt{n_1^2+n_2^2}.
		\end{align*}
		Therefore, in this case $F_A(n_1,n_2)$ is not a polynomial function. 	
	\end{example}

	\subsection{Global Newton-Okounkov bodies}
	Here we study the function $F_A(\bn)$ for an arbitrary algebra of almost integral type and we relate it to the construction of certain ``global Newton-Okounkov bodies''.
	We continue using \autoref{setup_mult_grad}.
	Since we may assume that $R$ is finitely generated over $\kk$ (see \autoref{rem_finite_gen}), we fix a valuation $\nu: \Quot(R) \to \ZZ^r$ that satisfies the conditions of \autoref{prop_bounded_val}; in particular, $\nu$ has leaves of bounded dimension.
	By using the fixed valuation $\nu: \Quot(R) \to \ZZ^r$, we make the following construction.
	
	\begin{definition}
		\label{def_glob_NO}
		Let $\Gamma_A$ be the valued semigroup of the $\NN^s$-graded algebra $A$:
		\[
		\Gamma_A \;:=\; \big\{(\nu(a),\bm)\in  \ZZ^{r}\times\NN^{s} \;\mid\; 0 \neq  a\in [A]_\bm \big\}.
		\]
		We define $\Delta(A) = \Con(\Gamma_A) \subset \RR^r \times \RR_{\geq 0}^s$ to be the closed convex cone generated by $\Gamma_A$ and we call it the \emph{global Newton-Okounkov body} of $A$.
	\end{definition}
	
	\begin{remark}
	    By abusing notation and following \cite{lazarsfeld09}, although $\Delta(A)$ is actually a cone, we call $\Delta(A)$ the global Newton-Okounkov body.
	\end{remark}
	
	We consider the following diagram 
	\begin{center}
		\begin{tikzpicture}
			\matrix (m) [matrix of math nodes,row sep=1.8em,column sep=2.5em,minimum width=2.5em, text height=1.5ex, text depth=0.25ex]
			{
				\Delta(A) &               & \RR^r \times \RR_{\ge 0}^s & \\
				&  \RR^r & & \RR_{\ge 0}^s\\				
			};
			\path[-stealth]
			(m-1-3) edge node [above]	{$\pi_1$} (m-2-2)
			(m-1-3) edge node [above]  {$\pi_2$} (m-2-4)
			;	
			\draw[right hook->] (m-1-1)--(m-1-3);			
		\end{tikzpicture}	
	\end{center}
	where $\pi_1  : \RR^{r} \times \RR_{\ge 0}^s \rightarrow \RR^r$ and $\pi_2  : \RR^{r} \times \RR_{\ge 0}^s \rightarrow \RR_{\ge 0}^s$ denote the natural projections.
	We denote the \emph{fiber} of the global Newton-Okounkov body $\Delta(A)$ over $x\in \RR_{\geq 0}^s$ by $\Delta(A)_x := \Delta(A) \cap \pi_2^{-1}(x)$.
	Since $A$ is a domain and $[A]_{\ee_i} \neq 0$ for all $1 \le i \le s$, it follows that $m(A^{(\bn)}) = 1$ and $\Delta(A^{(\bn)}) \subset \RR^r \times \{1\} \subset \RR^r \times \RR_{\ge 0}$.
	The following theorem describes the Newton-Okounkov bodies of the graded $\kk$-algebra $A^{(\bn)}$ (see \autoref{eq_veronese_subalg}).
	
	\begin{theorem}
		\label{thm-globalNO}
		Assume \autoref{setup_mult_grad}.
		The fiber $\Delta(A)_{\bn}$ of the global Newton-Okounkov body $\Delta(A) \subset \RR^r \times \RR_{\ge 0}^s$  coincides with the Newton-Okounkov body $\Delta(A^{(\bn)})$ for each $\bn\in \ZZ_+^s$, that is:
		$$
		\pi_1\left(\Delta(A)_\bn\right) \times \{1\} \;= \; \pi_1 \left(\Delta(A) \cap \pi_2^{-1}(\bn) \right) \times \{1\} \;=\; \Delta(A^{(\bn)}) \; \subset \; \RR^r \times \{1\}.
		$$
	\end{theorem}
	
	For the proof of \autoref{thm-globalNO} it is convenient to use the following well-known result.
	\begin{lemma}
		\label{lem-NObody-closure}
		Let $S\subset \ZZ^{r} \times \NN$ be a strongly non-negative semigroup. The Newton-Okounkov body $\Delta(S)$ can be computed as 
		$$
		\Delta(S) \;=\; \overline{\left\{\left(\frac{\bn}{n},m\right)\in \RR^{r}\times\RR_{\ge 0} \;\mid\; (\bn,nm) \in S \;\text{ for all }\; n > 0 \right\}},
		$$
		where $m = m(S)$.
	\end{lemma}
	\begin{proof}[Proof of \autoref{thm-globalNO}]
		Fix a vector $\bn\in \ZZ_+^s$ of positive integers. 
		The statement follows from the fact that the valued semigroup of $A^{(\bn)}$ is given by
		$$
		\Gamma_{A^{(\bn)}}  \;=\; \big\lbrace (\bm, n) \in \ZZ^r \times \NN \;\mid\; (\bm, n\bn) \in  \Gamma_A \big\rbrace.
		$$
		Hence, by \autoref{lem-NObody-closure} we have that
		\begin{eqnarray*}
			\Delta(A^{(\bn)}) &=& \overline{\left\{\left(\frac{\bm}{n}, 1\right) \in  \RR^{r}\times \RR_{\ge 0} \;\mid\; (\bm, n\bn) \in  \Gamma_A \;\text{ for all }\; n > 0 \right\}} \\
			&=& \pi_1 \left(  \overline{\left\{\left(\frac{\bm}{n},\bn\right)\in \RR^{r}\times \RR_{\ge 0}^s  \;\mid\; (\bm, n\bn) \in  \Gamma_A \;\text{ for all }\; n > 0 \right\}} \right) \times \{ 1 \} \\ 
			&=& \pi_1 \left(\Delta(A) \cap \pi_2^{-1}(\bn) \right) \times \{1\} \\
			&=& 	\pi_1\left(\Delta(A)_\bn\right) \times \{1\},
		\end{eqnarray*}
		which completes the proof.
	\end{proof}
	
	The following lemma provides a required uniformity result for the algebras $A^{(\bn)}$.
	
	\begin{lemma}
		\label{lem_uniform_veronese}
		Let $\bn \in \ZZ_+^s$ and $\bm \in \ZZ_+^s$.
		Then, one has $\ind(A^{(\bn)}) = \ind(A^{(\bm)})$ and $\ell_{A^{(\bn)}} = \ell_{A^{(\bm)}}$.
	\end{lemma}
	\begin{proof}
		For each $n \ge 0$ we have a canonical multiplication map
		$$
		\left[A^{(\bn)}\right]_n \otimes_\kk \left[A^{(\bm)}\right]_n \rightarrow \left[A^{(\bn+\bm)}\right]_n, \quad a \otimes_\kk a' \mapsto aa'.
		$$
		Then, by proceeding similarly to \autoref{prop_properties_semig}, we obtain that $\ind(A^{(\bn)}) \ge \ind(A^{(\bn+\bm)})$ and $\ell_{A^{(\bn)}} \le \ell_{A^{(\bn+\bm)}}$.
		
		Take $k > 0$ big enough such that $(k-1)\cdot\bn > \bm$.
		As above, by considering the multiplication maps 	$
		\big[A^{(\bn+\bm)}\big]_n \otimes_\kk \big[A^{((k-1)\cdot\bn-\bm)}\big]_n \rightarrow \big[A^{(k \cdot \bn)}\big]_n,
		$
		we get $\ind(A^{(\bn+\bm)}) \ge \ind(A^{(k\cdot\bn)})$ and $\ell_{A^{(\bn+\bm)}} \le \ell_{A^{(k \cdot \bn)}}$.
		Since $\ind(A^{(\bn)}) = \ind(A^{(k\cdot\bn)})$ and $\ell_{A^{(\bn)}} = \ell_{A^{(k \cdot \bn)}}$, it follows that $\ind(A^{(\bn)}) = \ind(A^{(\bn+\bm)})$ and $\ell_{A^{(\bn)}} = \ell_{A^{(\bn+\bm)}}$.
		
		Symmetrically, we also have that $\ind(A^{(\bm)}) = \ind(A^{(\bn+\bm)})$ and $\ell_{A^{(\bm)}} = \ell_{A^{(\bn+\bm)}}$.
	\end{proof}
	
	\begin{definition}
		\label{def_uniform_A}
		By using \autoref{lem_uniform_veronese}, we set:
		\begin{enumerate}[(i)]
			\item $\ind(A)$ be the index of $A$ which is  the constant value of $\ind(A^{(\bn)})$ for all $\bn \in \ZZ_+^s$.
			\item $\ell_A$ be the maximal dimension of leaves of $A$ which is the constant value of $\ell_{A^{(\bn)}}$ for all $\bn \in \ZZ_+^s$.
		\end{enumerate}  
	\end{definition}
	
	The corollary below gives an important characterization of the volume function of $A$ in terms of the global Newton-Okounkov body $\Delta(A)$.
	
	\begin{corollary}\label{cor-global}
		Assume \autoref{setup_mult_grad}.
		There exists a unique continuous homogeneous function of degree $q$ extending the volume function $F_A(\bn)$ defined in \autoref{eq_function_A} to the positive orthant $\RR_{\geq 0}^s$.
		This function is given by 
		$$
		F_A : \RR_{\geq 0}^s\to \RR, \qquad x \mapsto \ell_A \cdot \frac{\Vol_q(\Delta(A)_x)}{\ind(A)}.
		$$
		Moreover, the function $F_A$ is log-concave:
		$$
		F_A(x+y)^{\frac1q} \; \geq \; F_A(x)^{\frac1q}+F_A(y)^{\frac1q}\quad\text{for all } x,y\in \RR_{\geq 0}^s.
		$$
	\end{corollary}
	\begin{proof}
		By combining \autoref{thm_limit_graded_domain}, \autoref{thm-globalNO} and \autoref{lem_uniform_veronese}, it follows that $F_A : \RR_{\geq 0}^s\to \RR, \; x \mapsto \ell_A \cdot \frac{\Vol_q(\Delta(A)_x)}{\ind(A)}$ extends the function $F_A(\bn)$ defined in \autoref{eq_function_A} for all $\bn \in \ZZ_+^s$.
		
		Now, the homogeneity follows from the fact that $\Delta(A)_{\lambda\cdot x} =\lambda \cdot\Delta(A)_x$. Since $\Delta(A)$ is convex we have $\Delta(A)_x+\Delta(A)_y \subset \Delta(A)_{x+y}$ and hence the log-concavity follows from Brunn-Minkowski inequalities for the volume of convex bodies.
	\end{proof}
	
	As our next example shows, in general, \autoref{cor-global} is the most general statement about volume function $F_A$ for some multigraded algebra $A$.
	\begin{example}[Every concave function is a volume function]
		\label{ex-universal-volume}
		Let $R=\kk[u]$ be a polynomial ring and consider the polynomial ring $R[t_1,\ldots, t_s]$.
		Let further $f:\RR^s_{\geq 0} \to \RR_{\geq0}$ be any non-negative,  homogeneous of degree $1$, concave function. We define a family of vector spaces indexed by $\fn\in \NN^s$:
		\[
		[A]_{\fn} = \left(\bigoplus_{0\leq i\leq f(\fn)}\kk \cdot u^i \right) \ft^\fn\subset \Big[R[t_1,\ldots,t_s]\Big]_{\fn}
		\]
		The algebra $A_f:=\bigoplus_{\fn\in \NN^s} [A]_\bn \subset R[t_1,\dots,t_s]$ is naturally an $\NN^s$-graded algebra. Since $f$ is concave and homogeneous, the global Newton-Okounkov body $\Delta(A_f)$ of $A_f$ is a cone in $\RR^{s+1}$ given by:
		\[
		\Delta(A_f) =\left\{(y,x)\in \RR \times \RR^s_{\geq 0} \,\mid\, 0\leq y\leq f(x)\right\}.
		\]
		In particular, we get $F_{A_f}(x)=f(x)$, for any $x\in \RR^s_{\geq 0}$.
	\end{example}
	
	However in some cases one can say more about function $F_A(x)$.
	Let $C\subset \RR^s$ be a convex cone and consider a family $\{\Delta_x\}_{x\in C}$ of convex bodies of dimension $q$ parametrised by $C$. We say $\{\Delta_x\}_{x\in C}$ is {\it linear} if $\Delta_{\lambda_1x+\lambda_2y} = \lambda_1\Delta_x+\lambda_2\Delta_y$ for any $x,y\in C$ and $\lambda_1,\lambda_2>0$. According to Minkowski's theorem (see \autoref{thm_Minkowski}), the volume of a linear family of convex bodies is a homogeneous polynomial. Thus we obtain the following proposition.
	
	\begin{proposition}
		Assume that the fibers of the global Newton-Okounkov body form a linear family of convex bodies. 
		Then the function $F_A(x)$ is a homogeneous polynomial of degree $q$. 
	\end{proposition}
	\begin{proof}
		Since
		\[
		F_A(x) = \ell_A\cdot \frac{\Vol_q(\Delta(A)_x)}{\ind(A)},
		\]
		the statement of the proposition follows from Minkowski's theorem and the fact that 
		$
			\Delta_{x} = x_1\Delta_{\ee_1}+\cdots+x_s\Delta_{\ee_s}
	   $
	   for all $x = (x_1,\ldots,x_s)\in \RR_{\geq 0}^s$.
	\end{proof}

	Another important case is when the global Newton-Okounkov body is a polyhedral cone. In this case, the multiplicity function is a piecewise polynomial with respect to some fan supported on the positive orthant. 
	
	\begin{proposition}\label{prop_polyhedral}
		Assume that the global Newton-Okounkov body $\Delta(A)$ is a polyhedral cone. Then there exists a fan $\Sigma$ supported on $\RR_{\geq 0}^s$ such that the function $F_A(x)$ is polynomial at each cone of $\Sigma$.
	\end{proposition}
	
	\begin{proof}
		Indeed, by \cite[Proposition 4.1]{kaveh-villella}, for any polyhedral cone $C\subset \RR^q\times \RR^s$ with $\pi_2(C)=\RR_{\geq 0}^s$, the family $\{C\cap \pi_2^{-1}(x)\}_{x\in C}$ of polytopes 
		is piecewise linear with respect to some fan $\Sigma$ supported on $\RR_{\geq 0}^s$. Therefore, the proposition follows from Minkowski's theorem.
	\end{proof}
	
	\begin{example}
	    \label{examp_funct_min}
		Let $R=\kk[u]$ be a polynomial ring and consider the polynomial ring $R[t_1,t_2]$. 
		Let $f\colon \RR_{\geq 0}^2\to \RR$ be a function defined by
		$
		f(x_1,x_2)=\min(x_1,x_2).
		$
		We define a family of vector spaces:
		\[
		[A]_{(n_1,n_2)} = \left(\bigoplus_{i=0}^{f(n_1,n_2)}\kk u^i \right) t_1^{n_1}t_2^{n_2}\subset \Big[R[t_1,t_2]\Big]_{(n_1,n_2)}
		\]
		The algebra $A=\bigoplus_{(n_1,n_2)\in \NN^2}\subset R[t_1,t_2]$ is naturally an $\NN^2$-graded algebra. The global Newton-Okounkov body of $A$ is a cone $C$ in $\RR^3$ generated by vectors $(1,0,0),(0,1,0),(1,1,1)$ and, as in \autoref{ex-universal-volume},
		$
		F_A(x_1,x_2)=f(x_1,x_2)=\min(x_1,x_2).
		$
	\end{example}

	\section{Multigraded algebras of almost integral type with decomposable grading}
	\label{sect_decomp_grad}
	
	In this section, we introduce and study the notion of mixed multiplicities for certain multigraded algebras of almost integral type. 
	We treat a family of algebras that we call \emph{algebras with decomposable grading}.
	Our approach is inspired by the methods used in \cite{MIXED_MULT_GRAD_FAM, MIXED_VOL_MONOM, cutkosky2019}.

	\begin{definition}
		An $\NN^s$-graded algebra $A$ is said to have a \emph{decomposable grading} if we have the equality
		$$
		[A]_{(n_1,n_2,\ldots,n_s)} = [A]_{n_1\ee_1} \cdot [A]_{n_2\ee_2}  \cdot \, \cdots \, \cdot [A]_{n_s\ee_s}  
		$$
		for all $(n_1,n_2,\ldots,n_s) \in \NN^s$.
	\end{definition}
	
	We now proceed to define the mixed multiplicities of a multigraded algebra $A$ of almost integral type with decomposable grading. 
	Here we extend \autoref{rem_standard_graded}: our approach relies on showing that the corresponding function $F_A(\bn)$ coincides with a polynomial $G_A(\bn)$ when $\bn \in \ZZ_+^s$. 
	We use this polynomial to define the mixed multiplicities of $A$.
	For the rest of this section we use the following setting.
	
	\begin{setup}
		\label{setup_decomp}
		Let $\kk$ be a field and $R$ be a $\kk$-domain.
		Let $A \subset R[t_1,\ldots,t_s]$ be an $\NN^s$-graded $\kk$-algebra of almost integral type with decomposable grading.
		Set $d = \dim(A)$ and $q = d-s$.
		We assume that $[A]_{\ee_i} \neq 0$ for all $1 \le i \le s$.
	\end{setup}

	For each $p \ge 1$, let 
	$$
	\widehat{A}_{[p]} \;:=\; \kk\left[ [A]_{p\ee_1},\ldots, [A]_{p\ee_s} \right] \subset A
	$$ 
	be the $\NN^s$-graded algebra generated by $[A]_{p\ee_1},\ldots,[A]_{p\ee_s}$, and denote by $\widetilde{A}_{[p]} := \bigoplus_{\bn \in \NN^s} \left[\widehat{A}_{[p]}\right]_{p\bn}$ the standard $\NN^s$-graded algebra obtained by regrading $\widehat{A}_{[p]}$.
	For $p \gg 0$, \autoref{thm_dim_subfinite_alg} and the fact that $A$ has a decomposable grading imply that $\dim\big(\widetilde{A}_{[p]}\big) = d$. 
	Then, by \autoref{rem_standard_graded} the function $F_{\widetilde{A}_{[p]}}(\bn)$ coincides with the homogeneous polynomial 
	\begin{equation}
		\label{eq_poly_A_p}
		G_{\widetilde{A}_{[p]}}(\bn) = \sum_{|\dd| = q} \, \frac{1}{\dd !} e\big(\dd; \widetilde{A}_{[p]}\big) \, \bn^\dd
	\end{equation}
	for all $\bn \in \ZZ_+^{s}$, where $q = d -s$ and $e\big(\dd; \widetilde{A}_{[p]}\big)$ denotes the mixed multiplicity of $\widetilde{A}_{[p]}$ of type $\dd \in \NN^s$.

	For each $a \ge 1$, we consider the \emph{$a$-truncation} $G^{[a]}(A) := \kk\left[\cup_{i=1}^s\cup_{j=1}^a [A]_{j\ee_i}\right] \subset A$ of $A$ which is the subalgebra generated by the graded components $[A]_{j\ee_i}$ with $1 \le i \le s, 1 \le j \le a$.
	The next proposition says that the $a$-truncations can be used to approximate $F_A(\bn)$.
	
	\begin{proposition}
		\label{prop_tuncation}
		Assume \autoref{setup_decomp}.
		For each $\bn = (n_1,\ldots,n_s) \in \ZZ_+^s$, we have the equality
		$$
		F_A(\bn) \; = \; \lim_{a \to \infty} F_{G^{[a]}(A)}(\bn).
		$$
	\end{proposition}
	\begin{proof}
		Fix $\bn = (n_1,\ldots,n_s) \in \ZZ_+^s$.
		To simplify notation, set $C = A^{(\bn)}$ and $D^{a} = \left(G^{[a]}(A)\right)^{(\bn)}$.
		Note that $F_A(\bn) = \lim_{n \to \infty} \frac{\dim_{\kk}([C]_n)}{n^q}$ and $F_{G^{[a]}(A)}(\bn) = \lim_{n \to \infty} \frac{\dim_{\kk}([D^{a}]_n)}{n^q}$.
		By following the same steps as in \autoref{subsect_R_domain}, we can define strongly non-negative semigroups 
		$$
		\Gamma_C^{(t)} \,\subset\, \NN^{r+1} \qquad \text{ and } \qquad \Gamma_{D^{a}}^{(t)} \,\subset\, \NN^{r+1},
		$$
		such that $\dim_{\kk}\left([C]_n\right) = \sum_{t=1}^{\ell_C} \#\left[\Gamma^{(t)}_C\right]_n$ and $\dim_{\kk}\left([D^{a}]_n\right) = \sum_{t=1}^{\ell_{D^a}} \#\left[\Gamma^{(t)}_{D^a}\right]_n$.
		Consider the integer $a' = \lfloor a / \max\{n_1,\ldots,n_s\} \rfloor$ (where $\lfloor \beta \rfloor$ denotes the floor function of $\beta \in \RR$) and the inclusions 
		$$
		n \star \left[\Gamma^{(t)}_C\right]_{a'} \; \subset \; \left[\Gamma_{D^a}^{(t)}\right]_{na'} \; \subset \; \left[\Gamma^{(t)}_C\right]_{na'}.
		$$
		Then, \autoref{thm_approx} implies that 
		$$
		\lim_{a \to \infty}\left( \lim_{n\to \infty} \frac{\#\big[\Gamma_{D^a}^{(t)}\big]_{n}}{n^q} \right) = \lim_{n\to \infty} \frac{\#\big[\Gamma_{C}^{(t)}\big]_{n}}{n^q},
		$$
		and so the result follows.
	\end{proof}

	The following proposition deals with the case when $A$ is also a finitely generated $\kk$-algebra.
	
	\begin{proposition}
		\label{prop_subalg_p}
		Assume \autoref{setup_decomp} with $A$ being finitely generated over $\kk$.
		For each $\bn = (n_1,\ldots,n_s) \in \ZZ_+^s$, we have the equality
		$$
		F_A(\bn) \; = \; \lim_{p \to \infty} \frac{F_{\widetilde{A}_{[p]}}(\bn)}{p^q}.
		$$
	\end{proposition}
	\begin{proof}
		Fix $\bn = (n_1,\ldots,n_s) \in \ZZ_+^s$.
		We define the graded $\kk$-algebra
		$$
		B \, :=\, \bigoplus_{n = 0}^\infty [A]_{n\ee_1}^{n_1} \cdot [A]_{n\ee_2}^{n_2} \cdot \, \cdots \, \cdot [A]_{n\ee_s}^{n_s}.
		$$
		For each $p \ge 1$, let $C^p$ be the graded $\kk$-algebra 
		$$
		C^p \,:=\, \Big(\widetilde{A}_{[p]}\Big)^{(\bn)} \,=\, \bigoplus_{n =0}^\infty \left[\widehat{A}_{[p]}\right]_{np\bn} = \bigoplus_{n=0}^\infty [A]_{p\ee_1}^{nn_1} \cdot [A]_{p\ee_2}^{nn_2} \cdot\, \cdots\, \cdot [A]_{p\ee_s}^{nn_s}.
		$$ 
		Once again, we can define strongly non-negative semigroups $\Gamma_{B}^{(t)} \subset \NN^{r+1}$ and $\Gamma_{C^p}^{(t)} \subset \NN^{r+1}$ such that $\dim_{\kk}\left([B]_n\right) = \sum_{t=1}^{\ell_B} \#\left[\Gamma^{(t)}_B\right]_n$ and $\dim_{\kk}\left([C^{p}]_n\right) = \sum_{t=1}^{\ell_{C^p}} \#\left[\Gamma^{(t)}_{C^p}\right]_n$.
		
		Since $\left([B]_p\right)^n = \left[C^p\right]_n \subset [B]_{np}$, we obtain the corresponding inclusions
		$$
		n \star \left[\Gamma^{(t)}_B\right]_{p} \; \subset \; \left[\Gamma_{C^p}^{(t)}\right]_{n} \; \subset \; \left[\Gamma^{(t)}_B\right]_{np}.
		$$
		By \autoref{thm_approx}, it follows that 
		$$
		\lim_{p \to \infty}\left( \frac{1}{p^q}\lim_{n\to \infty} \frac{\#\big[\Gamma_{C^p}^{(t)}\big]_{n}}{n^q} \right) = \lim_{n\to \infty} \frac{\#\big[\Gamma_{B}^{(t)}\big]_{n}}{n^q}.
		$$
		As a consequence, we get $\lim_{p \to \infty} \frac{F_{\widetilde{A}_{[p]}}(\bn)}{p^q} = \lim_{n\to \infty} \frac{\dim_{\kk}\big([B]_{n}\big)}{n^q}$.
		To finish the proof, it remains to show the equality $F_A(\bn) = \lim_{n\to \infty} \frac{\dim_{\kk}\big([B]_{n}\big)}{n^q}$.
		
		As $A$ has a decomposable grading, the algebras $A^{(\ee_1)},\ldots,A^{(\ee_s)}$ are also finitely generated over $\kk$.
		Hence, by \cite[Lemma 13.10]{GORTZ_WEDHORN} we can choose $h > 0$ such that 
		$$
		[A]_{nh\bn} = [A]_{nn_1h\ee_1} \cdot [A]_{nn_2h\ee_2} \cdot \cdots \cdot [A]_{nn_sh\ee_s} = [A]_{nh\ee_1}^{n_1} \cdot [A]_{nh\ee_2}^{n_2} \cdot \, \cdots \, \cdot [A]_{nh\ee_s}^{n_s} = [B]_{nh}
		$$
		for all $n \ge 0$.
		We then obtain
		$$
		F_A(\bn)  = \lim_{n\to \infty} \frac{\dim_{\kk}\big([A]_{nh\bn}\big)}{n^qh^q} = \lim_{n\to \infty} \frac{\dim_{\kk}\big([B]_{nh}\big)}{n^qh^q}  = \lim_{n\to \infty} \frac{\dim_{\kk}\big([B]_{n}\big)}{n^q}.
		$$
		So, the proof of the proposition is complete.
	\end{proof}
	
	The next theorem contains the main result of this section.
	It shows that $F_A(\bn)$ is a polynomial like function when $A$ has a decomposable grading. 
	
	\begin{theorem}
		\label{thm_poly_decomp}
		Assume \autoref{setup_decomp}.
		Then, there exists a homogeneous polynomial $G_A(\bn) \in \RR[n_1,\ldots,n_s]$ of degree $q$ with non-negative real coefficients such that 
		$$
		F_A(\bn) \; = \; G_A(\bn) \quad \text{ for all } \quad  \bn \in \ZZ_+^s.
		$$
		Additionally, we have 
		$$
		G_A(\bn)  \; = \; \lim_{p \to \infty} \frac{G_{\widetilde{A}_{[p]}}(\bn)}{p^q} \; = \; \sup_{p \in \ZZ_+} \frac{G_{\widetilde{A}_{[p]}}(\bn)}{p^q} \quad \text{ for all } \quad  \bn \in \ZZ_+^s.
		$$
	\end{theorem}
	\begin{proof}
		Fix $\bn = (n_1,\ldots,n_s) \in \ZZ_+^s$.
		For any $a \ge 1$ and $p \ge 1$, we have the following inequalities
		$$
		\lim_{p \to \infty} \frac{F_{\widetilde{B^a}_{[p]}}(\bn)}{p^q} 
		\; \le \; \lim_{p \to \infty} \frac{F_{\widetilde{A}_{[p]}}(\bn)}{p^q} 
		\; \le \; F_A(\bn)
		$$
		where $B^a = G^{[a]}(A)$.
		Therefore, by combining \autoref{prop_tuncation} and \autoref{prop_subalg_p}, we obtain the equalities 
		$$
		F_A(\bn) \; = \; \lim_{p \to \infty} \frac{F_{\widetilde{A}_{[p]}}(\bn)}{p^q} \; = \; \sup_{p \in \ZZ_+} \frac{F_{\widetilde{A}_{[p]}}(\bn)}{p^q}.
		$$
		From \autoref{rem_standard_graded}, the function  $F_{\widetilde{A}_{[p]}}(\bn)$	coincides with the polynomial $G_{\widetilde{A}_{[p]}}(\bn) = \sum_{|\dd| = q} \frac{1}{\dd!}e\big(\dd; \widetilde{A}_{[p]}\big)\bn^\dd$ in \autoref{eq_poly_A_p} for all $\bn \in \ZZ_+^s$.
		It then necessarily follows that, for all $\bn \in \ZZ_+^s$, $F_A(\bn) = G_A(\bn)$ where $G_A(\bn) \in \RR[n_1,\ldots,n_s]$ is the polynomial 
		\begin{align}
			\label{eq_poly_G_A}
			\begin{split}
				G_A(\bn)  &\; = \; \lim_{p \to \infty} \frac{G_{\widetilde{A}_{[p]}}(\bn)}{p^q} \;=\; \sum_{|\dd| = q} \left(\lim_{p \to \infty}  \frac{e\big(\dd; \widetilde{A}_{[p]}\big)}{p^q}\right) \, \frac{\bn^\dd}{\dd!}\\
				&\; = \; \sup_{p \in \ZZ_+} \frac{G_{\widetilde{A}_{[p]}}(\bn)}{p^q} \;=\; \sum_{|\dd| = q} \left(\sup_{p \in \ZZ_+}  \frac{e\big(\dd; \widetilde{A}_{[p]}\big)}{p^q}\right) \, \frac{\bn^\dd}{\dd!}
			\end{split}
		\end{align}
		(see, e.g.,~\cite[Lemma 3.2]{cutkosky2019}).
		So, the result of the theorem follows.
	\end{proof}
	
	
	Below we have an example of a family of algebras with decomposable grading.
	
	\begin{example}
	    \label{examp_cox_ring}
		Let $G$ be a complex semisimple group and let $B\subset G$ be a Borel subgroup with a character lattice $M$. Denote by $A$ the Cox ring of $G/B$, i.e.
		\[
		A = \bigoplus_{L\in \Pic(G/B)} \HH^0(G/B, L),
		\]
		with a product induced by the natural maps:
		\[
		\HH^0(G/B, L_1)\otimes \HH^0(G/B, L_2) \to \HH^0(G/B, L_1\otimes L_2).
		\]
		The Cox ring $A$ is naturally an algebra graded by the Picard group $\Pic(G/B)$ of the flag variety $G/B$. By Borel's theorem, $\Pic(G/B)$ is isomorphic to $M$, for a character $\lambda\in M$, we will denote by $L_\lambda\in \Pic(G/B)$ the corresponding line bundle. Moreover, by Borel-Weil theorem one has
		\[
		\HH^0(G/B, L_\lambda) = \begin{cases} V_\lambda \quad \text{ if } \lambda \text{ is a dominant weight}\\
			0 \quad \;\;\;\;\text{otherwise.}
		\end{cases}
		\]
		Therefore, the Cox ring $A$ of $G/B$ is given by
		\[
		A=\bigoplus_{\lambda\in \Lambda^+} V_\lambda,
		\]
		where $\Lambda^+$ is the positive Weyl chamber, i.e.,~the direct sum is over dominant weights. 
		
		The product maps $[A]_\lambda\otimes [A]_\mu\to [A]_{\lambda+\mu}$ can be described in the following way. The tensor product $[A]_\lambda\otimes [A]_\mu = V_\lambda\otimes V_\mu$ decomposes into a direct sum of irreducible representations of $G$ with $V_{\lambda+\mu}$ appearing with multiplicity one. Then the multiplication map $[A]_\lambda\otimes [A]_\mu\to [A]_{\lambda+\mu}$ is the projection on $V_{\lambda+\mu}$ in the above decomposition. In particular, we have 
		\[
		[A]_{\sum n_i\omega_i} = [A]_{n_1\omega_1}\cdot \, \cdots \,\cdot [A]_{n_s\omega_s},
		\]
		where $(n_1,\ldots,n_s)\in \NN^s$ and $\omega_1,\ldots, \omega_s$ are fundamental weights of $G$. 
		Therefore, the Cox ring $A$ has decomposable grading and by \autoref{thm_poly_decomp} the volume function $F_A$ is a polynomial.
		
		Finally the global Newton-Okounkov body of $A$ has a nice description. By \cite{kaveh_crystal} there exists a valuation $\nu$ on $A$ such that for any dominant weight $\lambda$, the Newton-Okounkov body $\Delta(A^{(\lambda)})$ is the string polytope $St_\lambda$. Therefore, the global Newton-Okounkov body $\Delta(A)$ is the weighted string cone, which is in particular a polyhedral cone (\cite{littelmanncones,BerZel}). 
		
		Note that, in general, string polytopes $St_\lambda$ provide only piecewise linear family of polytopes on the positive Weyl chamber. So a priory by \autoref{prop_polyhedral}, the function $F_A$ is only piecewise polynomial with respect to some fan decomposition of the positive Weyl chamber. However, using virtual polytopes, one can construct a linear family of \emph{virtual string polytopes} which makes the polynomiality of $F_A$ evident. See \cite[Section 10]{hofscheier} for more details.
	\end{example}
	
	With \autoref{thm_poly_decomp} in hand we are now able to define the mixed multiplicities of $A$.
	
	\begin{definition}
		\label{def_mixed_mult_decomp}
		Assume \autoref{setup_decomp} and let $G_A(\bn)$ be as in \autoref{thm_poly_decomp}. 
		Write 
		$$
		G_A(\bn) \;=\; \sum_{|\dd| = q} \frac{1}{\dd!}\, e(\dd;A)\, \bn^\dd.
		$$
		For each  $\dd = (d_1,\ldots,d_s) \in \NN^s$ with $|\dd| = q$, we define the non-negative real number 
		$
		e(\dd;A) \ge 0
		$ 
		to be the {\it mixed multiplicity of type $\dd$ of $A$}.
	\end{definition}
	
	The next straightforward corollary shows that the mixed multiplicities $e(\dd; A)$ of $A$ can be expressed as a limit that depends on the multiplicities $e\big(\dd; \widetilde{A}_{[p]}\big)$ of the standard multigraded algebras $\widetilde{A}_{[p]}$.
	It can be seen as an extension into a multigraded setting of the Fujita approximation theorem for graded algebras given in \cite[Theorem 2.35]{KAVEH_KHOVANSKII}.
	
	\begin{corollary}
		\label{cor_vol_mult}
		Assume \autoref{setup_decomp}.
		Then, the following equalities hold
		$$
		e(\dd;A) 
		\; = \; 
		\lim_{p \to \infty}  \frac{e\big(\dd; \widetilde{A}_{[p]}\big)}{p^q}
		\; = \; 
		\sup_{p \in \ZZ_+}  \frac{e\big(\dd; \widetilde{A}_{[p]}\big)}{p^q}.
		$$
	\end{corollary}
	\begin{proof}
		It follows directly from \autoref{eq_poly_G_A}.
	\end{proof}
	
	Finally, we provide a complete characterization for the positivity of the mixed multiplicities of a multigraded algebra of almost integral type with decomposable grading. 
	This result is a direct consequence of \autoref{cor_vol_mult} and the general criterion of \cite{POSITIVITY}.
	
	Following the notation of \cite{POSITIVITY}, 	for an $\NN^s$-graded algebra $T$ and for each subset $\mathfrak{J} = \{j_1,\ldots,j_k\}  \subseteq \{1, \ldots, s\}$ denote by $T_{(\fJ)}$ the  $\NN^k$-graded $\kk$-algebra given by 
	$$
	T_{(\fJ)} := \bigoplus_{\substack{i_1\ge 0,\ldots, i_s\ge 0\\ i_{j} = 0 \text{ if } j \not\in \fJ}} {\left[T\right]}_{(i_1,\ldots,i_s)} \; \subset \; T.
	$$
	We obtain a full characterization for the positivity of the mixed multiplicities $e(\dd, A)$ of $A$ in terms of the dimensions $\dim\left(A_{(\fJ)}\right)$.
	
	\begin{theorem}
		\label{thm_postivity_decomp}
		Assume \autoref{setup_decomp}.
		Let $\dd=(d_1,\ldots,d_s) \in \NN^s$ such that $\lvert \dd \rvert=q$.
		Then, $e(\dd, A) > 0$ if and only if for each $\fJ = \{j_1,\ldots,j_k\} \subseteq \{1,\ldots,s\}$ the inequality 
		$$
		d_{j_1} + \cdots + d_{j_k} \;\le\; \dim\left(A_{(\fJ)}\right) - k
		$$
		holds.
	\end{theorem}
	\begin{proof}
		For each $p \ge 1$, \cite[Theorem B]{POSITIVITY} characterizes the positivity of the mixed multiplicities of the standard $\NN^s$-graded algebra $\widetilde{A}_{[p]}$, namely:
		$$
		e\left(\dd;\widetilde{A}_{[p]}\right) > 0 \;\;\, \Longleftrightarrow\;\;\, d_{j_1} + \cdots + d_{j_k} \;\le\; \dim\left(\left(\widetilde{A}_{[p]}\right)_{(\fJ)}\right) - k \;\text{ for each }\; \fJ = \{j_1,\ldots,j_k\} \subseteq \{1,\ldots,s\}.
		$$
		Since $A$ is an algebra with decomposable grading, by \autoref{thm_dim_subfinite_alg}, we can choose $p$ big enough such that 
		$$
		\dim\left(A_{(\fJ)}\right) \; = \; \dim\left(\left(\widetilde{A}_{[p]}\right)_{(\fJ)}\right)
		$$
		for all $\fJ = \{j_1,\ldots,j_k\} \subseteq \{1,\ldots,s\}$.
		Therefore, the result follows from \autoref{cor_vol_mult}.
	\end{proof}

	\section{Application to graded families of ideals}
		\label{sect_appl_ideals}
	
	In this section, we apply the results of \autoref{sect_decomp_grad} to the case of graded families of ideals and we recover some results from \cite{MIXED_MULT_GRAD_FAM, MIXED_VOL_MONOM, cutkosky2019}.
	We also obtain a characterization for the positivity of mixed multiplicities of certain graded families of ideals.
	First, we recall the notion of mixed multiplicities introduced by Bhattacharya in \cite{Bhattacharya} for the case of ideals, and extended for (not necessarily Noetherian) graded families of ideals in \cite{MIXED_MULT_GRAD_FAM}.
	
	Let $\kk$ be a field and $R$ be a finitely generated positively graded $\kk$-domain.
	We denote the graded irrelevant ideal of $R$ by $\mm = [R]_+ = \bigoplus_{n = 1}^\infty [R]_n \subset R$. 
	Assume that $R$ has positive dimension $d = \dim(R) > 0$.
	
	\begin{remark}
		\label{rem_mixed_mult_ideals}
		Let $I, J_1, \ldots, J_s$ be non-zero homogeneous ideals in $R$ such that  $I$ is $\mm$-primary.  
		Then,  for $n_0\gg 0$ and $\bn=(n_1,\ldots,n_s)\gg \mathbf{0}$ the Bhattacharya function $\dim_\kk(I^{n_0}J_1^{n_1}\cdots J_s^{n_s}/I^{n_0+1}J_1^{n_1}\cdots J_s^{n_s})$ coincides with  a polynomial  of total degree $d-1$  whose homogeneous term  in degree $d-1$ can be written as 
		$$
		B_{(I;J_1,\ldots, J_s)}(n_0,n_1,\ldots, n_s) \;:=\; \sum_{\substack{
				(d_0,\fd)=(d_0,d_1,\ldots, d_s)\in \NN^{s}\\ d_0+|\fd|=d-1}
		}\frac{ e_{(d_0,\fd)}(I\mid J_1,\ldots, J_s)}{d_0!d_1!\cdots d_s!} n_0^{d_0}n_1^{d_1}\cdots, n_s^{d_s}.
		$$
		Using standard techniques (see \cite[proof of Lemma 4.2]{MIXED_VOL_MONOM}), one may show that for each $n_0\in \NN$ and $\bn=(n_1,\ldots, n_s)\in \NN^s$ the limit 
		$$
		\lim_{n \to \infty}\frac{\dim_\kk \left( J_1^{nn_1}\cdots J_s^{nn_s}/I^{nn_0}J_1^{nn_1}\cdots J_s^{nn_s} \right)}{n^{d}}
		$$ 
		exists and coincides with the following  polynomial 
		\begin{equation*}
			G_{(I;J_1,\ldots, J_s)}(n_0,n_1,\ldots, n_s) \;:=\; \sum_{(d_0,\dd) \in \NN^{s+1}, d_0+|\fd|=d-1}\frac{ e_{(d_0,\fd)}(I\mid J_1,\ldots, J_s)}{(d_0+1)!d_1!\cdots d_s!}\, n_0^{d_0+1}n_1^{d_1}\cdots n_s^{d_s}.
		\end{equation*}
		The numbers $e_{(d_0,\fd)}(I\mid J_1,\ldots, J_s)$ are non-negative integers called the {\it mixed multiplicities} of $J_1,\ldots, J_s$ with respect to $I$.	
	\end{remark}
	
	A sequence of ideals $\II=\{I_n\}_{n\in \NN}$ is a {\it graded family} if $I_0=R$ and $I_iI_j\subseteq I_{i+j}$ for every $i,j\in \NN$. 
	The graded family is {\it Noetherian} if the corresponding Rees algebra $R[\II t]=\bigoplus_{n\in \NN}I_nt^n\subseteq R[t]$ is Noetherian. 
	The graded family is \emph{$\mm$-primary} when each $I_n$ is $\mm$-primary, and it is a \emph{filtration} when  $I_{n+1}\subseteq I_n$ for every $ n \in \NN$.  
	For a homogeneous ideal $J \subset R$ we denote ${\rm maxdeg}(J) := \max\{j\mid \left[J\otimes_R R/\mm\right]_j\neq 0\}$, that is, the maximum degree of a minimal set of homogeneous generators of $J$. 
	
	Throughout this section, we assume the following setup. 
	
	\begin{setup}
		\label{setup_grad_fam_ideals}
		Let $\kk$ be a field and $R$ be a finitely generated positively graded $\kk$-domain.
		Let $\II = \{I_n\}_{n \in \NN}$ be a (not necessarily Noetherian) graded family of $\mm$-primary homogeneous ideals in $R$.
		Let $\JJ(1)=\{J(1)_n\}_{n\in \NN}$, $\ldots$, $\JJ(s)=\{J(s)_n\}_{n\in \NN}$ be (not necessarily Noetherian) graded families of non-zero homogeneous ideals in $R$.
		We assume that there exists $\beta \in \NN$ satisfying 
		\begin{equation}
			\label{eq_linear_growth}
			{\rm maxdeg}(J(i)_n) \le \beta n\quad\text{for all}\quad 1 \le i \le s\ \ \text{and}\ \ n \in \NN.
		\end{equation}
	\end{setup}
	
	\begin{remark}
		The condition \autoref{eq_linear_growth} is automatically satisfied in the case of $\mm$-primary graded families of ideals.
		Similar assumptions to the one in \autoref{eq_linear_growth} have been considered in previous works regarding limits of graded families of ideals \cite[Theorem 6.1]{cutkosky2014}, \cite{MIXED_VOL_MONOM}, \cite{MIXED_MULT_GRAD_FAM}.
	\end{remark}
	
	To extend the discussions of \autoref{rem_mixed_mult_ideals}, we need to study the following function 
	\begin{equation}
		\label{eq_funct_grad_fam_ideals}
		F_{(\II;\JJ(1),\ldots, \JJ(s))}(n_0,\bn) \; := \; \lim_{n \to \infty}\frac{\dim_\kk \left( J(1)_{nn_1}\cdots J(s)_{nn_s}/I_{nn_0}J(1)_{nn_1}\cdots J(s)_{nn_s} \right)}{n^{d}} 
	\end{equation}
	for all $n_0 \in \NN_+$ and $\bn = (n_1,\ldots,n_s) \in \NN_+^s$.
	For a vector $\bn=(n_1,\ldots, n_s) \in \NN^{s}$, we  abbreviate $\bJ_\bn=J(1)_{n_1}\cdots J(s)_{n_s}$. 
	By \cite[Lemma 3.9]{MIXED_VOL_MONOM}, there exists $c > \beta$ such that 
	$$
	\mm^{c(n_0+|\bn|)} \, \cap \, \bJ_\bn \;=\; \mm^{c(n_0+|\bn|)} \, \cap \, I_{n_0}\bJ_\bn
	$$
	for all $n_0 \in \NN$ and $\bn = (n_1,\ldots,n_s) \in \NN^s$.
	Then, we have the equality
	\begin{align}
		\label{eq_F_filt_F_diff}
		\begin{split}
			F_{(\II;\JJ(1),\ldots, \JJ(s))}(n_0,\bn) \; := \; & \lim_{n \to \infty}  \dim_{\kk}\left(\bJ_{n\bn} / \big(\mm^{cn(n_0+|\bn|)+1} \, \cap \, \bJ_{n\bn}\big) \right) \big/ n^d \\
			& \qquad\quad - \; \lim_{n \to \infty}  	\dim_{\kk}\left(I_{nn_0}\bJ_{n\bn} / \big(\mm^{cn(n_0+|\bn|)+1} \, \cap \, I_{nn_0}\bJ_{n\bn}\big) \right) \big/ n^d
		\end{split}
	\end{align} 
	for all $n_0 \in \NN$ and $\bn = (n_1,\ldots,n_s) \in \NN^s$.
	We consider the multi-Rees algebras 
	$$
	\Rees(\II, \JJ(1),\ldots,\JJ(s)) := \bigoplus_{n_0\in \NN, \bn \in \NN^s} I_{n_0} \bJ_{\bn} t_0^{n_0}\ttt^\bn
	\qquad \text{ and } \qquad \Rees(\mathfrak{R}, \JJ(1),\ldots,\JJ(s)) := \bigoplus_{n_0 \in \NN, \bn \in \NN^s} \bJ_{\bn} t_0^{n_0}\ttt^\bn
	$$
	where $\ttt^\bn = t_1^{n_1}\cdots t_s^{n_s}$ and $\mathfrak{R}$ is the trivial filtration of identity ideals, i.e.,~$\mathfrak{R} = \{\mathfrak{R}_n\}_{n \in \NN}$ with $\mathfrak{R}_n = R$.
	We have the corresponding $\NN^{s+1}$-graded subalgebras 
	$$
	A \;:=\;  \bigoplus_{n_0\in \NN, \bn \in \NN^s} \left( \bigoplus_{k = 0}^{c(n_0+|\bn|)}  {\big[I_{n_0} \bJ_{\bn}\big]}_k  \right) t_0^{n_0}\ttt^\bn \;\; \subset \;\;  \Rees(\II, \JJ(1),\ldots,\JJ(s))
	$$
	and 
	$$
	B \;:=\;  \bigoplus_{n_0\in \NN, \bn \in \NN^s} \left( \bigoplus_{k = 0}^{c(n_0+|\bn|)}  {\big[ \bJ_{\bn}\big]}_k  \right) t_0^{n_0}\ttt^\bn \;\; \subset \;\;  \Rees(\mathfrak{R}, \JJ(1),\ldots,\JJ(s)).
	$$
	Notice that both $A$ and $B$ are of almost integral type and have decomposable gradings.
	Since $\dim(A) = \dim(B) = d+s+1$, we can rewrite \autoref{eq_F_filt_F_diff} in terms of the volume functions of $A$ and $B$ (see \autoref{eq_function_A}), that is 
	\begin{equation}
		\label{eq_F_ideals_A_B}
		F_{(\II;\JJ(1),\ldots, \JJ(s))}(n_0,\bn) \; = \; F_B(n_0,\bn) - F_A(n_0,\bn).
	\end{equation}
	Therefore, as a simple consequence of \autoref{thm_poly_decomp} we obtain the following result, which extends \autoref{rem_mixed_mult_ideals} and allows us to define mixed multiplicities for graded families of ideals.
	
	\begin{theorem}
		\label{thm_grad_fam_ideals}
		Assume \autoref{setup_grad_fam_ideals}.
		Then, there exists a homogeneous polynomial $G_{(\II;\JJ(1),\ldots, \JJ(s))}(n_0,\bn) \in \RR[n_0,n_1,\ldots,n_s]$ of degree $d$ with non-negative real coefficients such that 
		$$
		F_{(\II;\JJ(1),\ldots, \JJ(s))}(n_0,\bn) \; = \; G_{(\II;\JJ(1),\ldots, \JJ(s))}(n_0,\bn) \quad \text{ for all } \quad  n_0 \in \ZZ_+, \, \bn \in \ZZ_+^s,
		$$
		where $F_{(\II;\JJ(1),\ldots, \JJ(s))}(n_0,\bn)$ is the function in \autoref{eq_funct_grad_fam_ideals}.
		Additionally, we have 
		$$
		G_{(\II;\JJ(1),\ldots, \JJ(s))}(n_0,\bn) \; = \; \lim_{p \to \infty} \frac{
			G_{(I_p;J(1)_p,\ldots, J(s)_p)}(n_0,\bn)
		}{p^d}  \quad \text{ for all }   \quad  n_0 \in \ZZ_+, \, \bn \in \ZZ_+^s,
		$$
		where $G_{(I_p;J(1)_p,\ldots, J(s)_p)}(n_0,\bn)$ is the corresponding polynomial of the ideals $I_p, J(1)_p,\ldots,J(s)_p$ (as in \autoref{rem_mixed_mult_ideals}).
	\end{theorem}
	\begin{proof}
		By using the equality \autoref{eq_F_ideals_A_B}, the result follows by applying \autoref{thm_poly_decomp} to the $\NN^{s+1}$-graded algebras of almost integral type $A$ and $B$ that have decomposable grading.
	\end{proof}
	
	As a consequence of \autoref{thm_grad_fam_ideals} we have the following definition.
	
	\begin{definition}
		Assume \autoref{setup_grad_fam_ideals} and let $G_{(\II;\JJ(1),\ldots, \JJ(s))}(n_0,\bn)$ be as in \autoref{thm_grad_fam_ideals}. 
		Write 
		$$
		G_{(\II;\JJ(1),\ldots, \JJ(s))}(n_0,\bn) \;=\; \sum_{(d_0,\dd) \in \NN^{s+1}, d_0+|\fd|=d-1}\frac{ e_{(d_0,\fd)}(\II\mid \JJ(1),\ldots, \JJ(s))}{(d_0+1)!d_1!\cdots d_s!}\, n_0^{d_0+1}\bn^\dd.
		$$
		The numbers $e_{(d_0,\fd)}(\II\mid \JJ(1),\ldots, \JJ(s))$ are non-negative integers called the {\it mixed multiplicities} of $\JJ(1),\ldots, \JJ(s)$ with respect to $\II$.
	\end{definition}

	The following result recovers the ``Volume = Multiplicity formula'' for the mixed multiplicities of graded families of ideals (see \cite{MIXED_MULT_GRAD_FAM, MIXED_VOL_MONOM}).
	
	\begin{corollary}
		\label{cor_ideals_vol_mult}
		Assume \autoref{setup_grad_fam_ideals}.
		Then, the following equalities hold
		$$
		e_{(d_0,\fd)}(\II\mid \JJ(1),\ldots, \JJ(s))
		\; = \; 
		\lim_{p \to \infty}  \frac{e_{(d_0,\fd)}(I_p\mid J(1)_p,\ldots, J(s)_p)}{p^d}.
		$$
	\end{corollary}
	\begin{proof}
		It follows from \autoref{thm_grad_fam_ideals} and \autoref{rem_mixed_mult_ideals}.
	\end{proof}

	\subsection{Positivity for graded families of equally generated ideals}
		\label{subsect_equigen_ideals}
	
	Here, we restrict to the following graded families of ideals and we provide a criterion for the positivity of their mixed multiplicities. 
	
	\begin{setup}
		\label{setup_grad_fam_equigen}
		Assume \autoref{setup_grad_fam_ideals}.
		Let $\bmm$ be the filtration $\bmm := \{\mm^n\}_{n \in \NN}$.
		For each $1 \le i \le s$, assume that there is an integer $h_i \ge 1$ such that the ideal $J(i)_n$ is generated in degree $nh_i$ for all $n \in \NN$,  that is, 
		$$
		J(i)_{n} \; = \; \left( \big[J(i)_{n}\big]_{nh_i} \right) \quad \text{ for all } n \in \NN.
		$$
	\end{setup}
	
	Our focus is on characterizing the positivity of the mixed multiplicities $e_{(d_0,\dd)}(\bmm \mid \JJ(1),\ldots,\JJ(s))$.
	We define the following $\NN^{s+1}$-graded $\kk$-algebra 
	\begin{equation}
		\label{eq_algebra_T}
		T \; := \; \bigoplus_{n_0,\ldots,n_s \in \NN}   \mm^{n_0} J(1)_{n_1} \cdots J(s)_{n_s} \Big/ \mm^{n_0+1} J(1)_{n_1} \cdots J(s)_{n_s}.
	\end{equation}
	Since each ideal $J(i)_{n_i}$ is generated in degree $n_ih_i$, by Nakayama's lemma we get the isomorphism 
	$$
	T \; \cong \; \bigoplus_{n_0,\ldots,n_s \in \NN}   \Big[\mm^{n_0}\Big]_{n_0} \Big[J(1)_{n_1}\Big]_{n_1h_1} \cdots \Big[J(s)_{n_s}\Big]_{n_sh_s} \quad \subset \quad \Rees(\bmm, \JJ(1),\ldots, \JJ(s)).
	$$
	As a consequence we obtain that $T$ is a $\kk$-domain of almost integral type with decomposable grading, and so by \autoref{thm_poly_decomp} and \autoref{def_mixed_mult_decomp}  we can consider its mixed multiplicities $e\big(d_0,\dd; T\big)$ for $d_0 \in \NN, \dd\in \NN^{s}$.
	We have that $\dim(T) = \dim(R)+s$.
	The following proposition shows that the mixed multiplicities of $T$ coincide with the mixed multiplicities $e_{(d_0,\dd)}(\bmm \mid \JJ(1),\ldots,\JJ(s))$.
	
	\begin{proposition}
		\label{prop_equal_mixed_mult_T}
		For each $d_0 \in \NN, \, \dd \in \NN^s$ with $d_0 + |\dd| = d-1$, we have the equality
		$$
		e_{(d_0,\dd)}(\bmm \mid \JJ(1),\ldots,\JJ(s)) \; = \; e\big(d_0,\dd; T\big).
		$$
	\end{proposition}
	\begin{proof}
		For each $p \ge 1$, we define the following standard $\NN^{s+1}$-graded algebra $$
		C^p \; = \; \bigoplus_{n_0,n_1, \ldots, n_s \in \NN^{s+1}}  {\Big[\kk\big[[T]_{\ee_0}, [T]_{p\ee_1},\cdots, [T]_{p\ee_s}\big]\Big]}_{(n_0,n_1p,\ldots,n_sp)}
		$$
		(i.e.,~the regrading of the algebra generated by the graded parts $[T]_{\ee_0}, [T]_{p\ee_1},\cdots, [T]_{p\ee_s}$).
		Notice that the polynomial $G_{C^p}$ of $C^p$ (see \autoref{rem_standard_graded}) equals the polynomial $B_{(\mm;J(1)_p,\ldots,J(s)_p)}$ corresponding to the ideals $\mm$ and $J(1)_p,\ldots,J(s)_p$ (see \autoref{rem_mixed_mult_ideals}), that is, 
		$$
		G_{C^p}(n_0,\bn) \; = \; B_{(\mm;J(1)_p,\ldots, J(s)_p)}(n_0,\bn).
		$$
		Hence $e\left(d_0,\dd; C^p\right) = e_{(d_0,\dd)}\left(\mm \mid J(1)_p,\ldots, J(s)_p\right)$. 
		
		Since the following polynomials are equal
		$$
		G_{C^p}(pn_0,\bn) \; = \; G_{\widetilde{T}_{[p]}}(n_0,\bn),
		$$ 
		by comparing the coefficients, we obtain  $e\big(d_0,\dd; \widetilde{T}_{[p]}\big)  = p^{d_0} e\left(d_0,\dd; C^p\right)$ (recall that $\widetilde{T}_{[p]}$ denotes the regrading of the algebra generated by the graded parts $[T]_{p\ee_0}, [T]_{p\ee_1},\cdots, [T]_{p\ee_s}$).
		Similarly, we have the equality of polynomials
		$$
		G_{(\mm;J(1)_p,\ldots, J(s)_p)}(pn_0,\bn) \; = \; G_{(\mm^p;J(1)_p,\ldots, J(s)_p)}(n_0,\bn) \qquad (\text{see \autoref{rem_mixed_mult_ideals}})
		$$
		that implies $e_{(d_0,\dd)}\left(\mm^p \mid J(1)_p,\ldots, J(s)_p\right) = p^{d_0+1}e_{(d_0,\dd)}\left(\mm \mid J(1)_p,\ldots, J(s)_p\right)$.
		
		By combining the above equalities with \autoref{cor_vol_mult} and \autoref{cor_ideals_vol_mult}, we get 
		\begin{align*}
			e_{(d_0,\dd)}(\bmm \mid \JJ(1),\ldots,\JJ(s)) \;  &=	 \; \lim_{p \to \infty}  \frac{e_{(d_0,\fd)}(\mm^p\mid J(1)_p,\ldots, J(s)_p)}{p^d}  \\
			&=	 \; \lim_{p \to \infty}  \frac{e_{(d_0,\fd)}(\mm\mid J(1)_p,\ldots, J(s)_p)}{p^{d-1-d_0}} \\
			&=	 \; \lim_{p \to \infty}  \frac{e\left(d_0,\dd; C^p\right)}{p^{d-1-d_0}} \\
			&=	 \; \lim_{p \to \infty}  \frac{e\big(d_0,\dd; \widetilde{T}_{[p]}\big)}{p^{d-1}} \\
			&= \; e\big(d_0,\dd; T\big).
		\end{align*}
		Therefore, the result follows.
	\end{proof}
	
	The \emph{analytic spread} of an ideal $I \subset R$ is given by $\ell(I) := \dim\Big(\Rees(I)/\mm\Rees(I)\Big)$, where $\Rees(I)$ denotes the Rees algebra of $I$.
	When the ideal $I = (f_1,\ldots,f_e) \subset R$ is generated by homogeneous elements of the same degree, one has that $\Rees(I)/\mm\Rees(I) \cong \kk[f_1,\ldots,f_e]$ and so $\ell(I) = \dim\left(\kk[f_1,\ldots,f_e]\right)$.
	The following positivity criterion extends the result of \cite[Theorem 4.4]{POSITIVITY}.
	It characterizes the positivity of the mixed multiplicities $e_{(d_0,\dd)}(\bmm \mid \JJ(1),\ldots,\JJ(s))$.
	
	\begin{theorem}
		\label{thm_postiv_grad_fam}
		Assume \autoref{setup_grad_fam_equigen}.
		Let $d_0 \in \NN, \, \dd \in \NN^s$ such that $d_0 + \lvert \dd \rvert = d-1$.
		Then, 
		$$
		e_{(d_0,\dd)}(\bmm \mid \JJ(1),\ldots,\JJ(s)) > 0
		$$ 
		if and only if for all $p\gg 0 $ and $\fJ = \{j_1,\ldots,j_k\} \subseteq \{1,\ldots,s\}$ the inequality 
		$$
		d_{j_1} + \cdots + d_{j_k} \;\le\; \ell\Big(J(j_1)_p \cdots J(j_k)_p \Big) - 1
		$$
		holds.
	\end{theorem}
	\begin{proof}
		By using \autoref{prop_equal_mixed_mult_T} and \autoref{cor_vol_mult}, to show $	e_{(d_0,\dd)}(\bmm \mid \JJ(1),\ldots,\JJ(s)) > 0$, it is enough to show that for $p\gg 0$ one has
		$$
		e_{(d_0,\fd)}\big(\mm\mid J(j_1)_p, \ldots, J(j_k)_p \big) = e\big(d_0,\dd; \widetilde{T}_{[p]}\big) > 0.
		$$
		For any $p > 0$, \cite[Theorem 4.4]{POSITIVITY} implies that the inequality $e_{(d_0,\fd)}\big(\mm\mid J(j_1)_p, \ldots, J(j_k)_p \big)  > 0$ is equivalent to the condition of having $
		d_{j_1} + \cdots + d_{j_k} \;\le\; \ell\Big(J(j_1)_p \cdots J(j_k)_p \Big) - 1
		$
		for all $\fJ = \{j_1,\ldots,j_k\} \subseteq \{1,\ldots,s\}$.
		So, the result follows.
	\end{proof}

	\subsection{Positivity for mixed volumes of convex bodies}	
		\label{subsect_mixed_vol}
	
	In this subsection, by exploiting the known close relation between mixed multiplicities and mixed volumes (see \cite{TRUNG_VERMA_MIXED_VOL, MIXED_VOL_MONOM}), we provide a positivity criteria for mixed volumes.
	Following the notation of \cite{MIXED_VOL_MONOM}, we use the setup  below.
	
	\begin{setup}
		\label{setup_mixed_volumes}
		Let $\bK = (K_1,\ldots,K_s)$ be a sequence of convex bodies in $\RR_{\ge 0}^d$.
		Let $\kk$ be a field, $R$ be the polynomial ring $R = \kk[x_1,\ldots,x_{d+1}]$, and  $\mm=\left(x_1,\ldots,x_{d+1}\right)$.
		Let 
		$
		\pi_1 : \RR^{d+1} \rightarrow \RR^{d} , \, (\alpha_1,\ldots,\alpha_d,\alpha_{d+1}) \mapsto (\alpha_1,\ldots,\alpha_d)
		$ be the projection into the first $d$ factors. 
		Let $\pi : \RR^{d+1} \rightarrow \RR$ the linear map 
		$
		\pi : \RR^{d+1} \rightarrow \RR, \, (\alpha_1,\ldots,\alpha_d,\alpha_{d+1}) \mapsto \alpha_1+\cdots+\alpha_d+\alpha_{d+1}.
		$			
		For $1 \le i \le s$, choose $h_i \in \NN$ a positive integer such that $K_i \subset \pi_1\big(\pi^{-1}(h_i) \cap \RR_{\ge 0}^{d+1}\big)$.
		The corresponding \emph{homogenization} of $K_i$ (with respect to $h_{i}$) is defined as the convex body 
		$$
		\widetilde{K_i} \;:=\; \left(K_i \times \RR\right) \cap \pi^{-1}(h_i) \; \subset \; \RR_{\ge 0}^{d+1}.
		$$
		Let $C_{K_i}$ be the corresponding cone $C_{K_i} := \text{Cone}(\widetilde{K_i})$.
		Consider the semigroup $S_{K_i} \subset \NN^{d+1}$ given by
		$$
		S_{K_i}:= C_{K_i} \cap \NN^{d+1} \cap \left(\bigcup_{k=1}^\infty \pi^{-1}\left(kh_i\right)\right).
		$$
		We consider the (not necessarily Noetherian) graded family of monomial ideals
		\begin{equation*}
			\JJ(i) \;:=\; \{J(i)_n\}_{n \in \NN},  \quad\text{where }\quad J(i)_n \;:=\; \Big(\bx^\bm \mid \bm \in S_{K_i} \text{ and } \lvert \bm \rvert = nh_{i}\Big) \subset R.  
		\end{equation*}
		Let $\bmm$ be the filtration $\bmm := \{\mm^n\}_{n \in \NN}$.
	\end{setup}

	As a direct consequence of our previous developments, we recover a classical criterion for the positivity of mixed volumes (see \cite[Theorem 5.1.8]{SCHNEIDER}).
	
	\begin{theorem}
		\label{thm_mixed_vol}
		Assume \autoref{setup_mixed_volumes}.
		Let $\dd \in \NN^s$ such that $ \lvert \dd \rvert = d$.
		Then, $\MV_d\left(\bK_{ \dd}\right)>0$ if and only if for each $\fJ = \{j_1,\ldots,j_k\} \subseteq \{1,\ldots,s\}$ the inequality 
		$$
		d_{j_1} + \cdots + d_{j_k} \;\le\; \dim_\RR\left(\sum_{i=1}^k K_{j_i}\right)
		$$
		holds. 
	\end{theorem}
	\begin{proof}
		From \cite[Theorem 5.4]{MIXED_VOL_MONOM} we have that $\MV_d\left(\bK_{ \dd}\right) = 	e_{(0,\dd)}(\bmm \mid \JJ(1),\ldots,\JJ(s))$.
		Hence, \autoref{thm_postiv_grad_fam} implies that $\MV_d\left(\bK_{ \dd}\right)>0$ if and only if for all $p\gg 0 $ and $\fJ = \{j_1,\ldots,j_k\} \subseteq \{1,\ldots,s\}$ the inequality 
		$$
		d_{j_1} + \cdots + d_{j_k} \;\le\; \ell\Big(J(j_1)_p \cdots J(j_k)_p \Big) - 1
		$$
		holds.
		For each ideal $J(j_1)_p \cdots J(j_k)_p$ there corresponds the convex body 
		$$
		K\big(J(j_1)_p \cdots J(j_k)_p\big) \;:=\; \pi_1\left(\conv
		\Big\lbrace
		\bm \in \NN^{d+1}\mid \bx^\bm \in  \big[J(j_1)_p \cdots J(j_k)_p\big]_{ph_{j_1}+\cdots+ph_{j_k}}  
		\Big\rbrace
		\right) \; \subset \; \RR_{\ge 0}^d.
		$$
		Notice that $\ell(J(j_1)_p \cdots J(j_k)_p) - 1 = \dim_\RR(K(J(j_1)_p \cdots J(j_k)_p))$.
		For $p \gg 0$, one obtains the equality $\dim_\RR(K(J(j_1)_p \cdots J(j_k)_p)) = \dim_\RR(\sum_{i=1}^k K_{j_i})$, as the convex bodies $K(J(j_1)_p \cdots J(j_k)_p)$ approximate the Minkowski sum $\sum_{i=1}^k K_{j_i}$.
		Therefore, the proof of the theorem is complete.
	\end{proof}

	\section{Multigraded linear series}\label{section_mult_lin_series}
	
	In this section, we apply our main results from the previous sections to the case of multigraded linear series. 
	Here, we extend the results of \cite{lazarsfeld09} on the multigraded linear series. 
	Our main improvement is the fact that we deal with valuations with leaves of bounded dimension instead of restricting to valuations with only one-dimensional leaves.
	Furthermore, with the family of multigraded linear series with decomposable grading, we provide an interesting family for which the volume function is a polynomial.

	Following the notation of \cite{HARTSHORNE}, we say that $X$ is a \emph{variety} over a field $\kk$ if $X$ is a reduced and irreducible separated scheme of finite type over $\kk$.

	Throughout this section the following setup is used. 
	
	\begin{setup}
		\label{setup_linear_series}
		Let $\kk$ be a field and $X$ be a proper variety over $\kk$.
		Let $D_1,\ldots,D_s$ be a sequence of Cartier divisors on $X$.
	\end{setup}
	
	We consider the section ring of the divisors $D_1,\ldots,D_s$, which given by 
	$$
	\SSS(D_1,\ldots,D_s) \; := \; \bigoplus_{(n_1,\ldots,n_s) \in \NN^s} \HH^0\big(X, \OO(n_1D_1+\cdots+n_sD_s)\big).
	$$
	Notice that $	\SSS(D_1,\ldots,D_s)$ is by construction an $\NN^s$-graded $\kk$-algebra.
	To simplify notation, for any $\bn = (n_1,\ldots,n_s) \in \NN^s$, we denote the divisor $n_1D_1+\cdots + n_sD_s$ by 
	$$
	\bn  D \; := \; n_1D_1+\cdots + n_sD_s.
	$$
	The following basic result shows that the section ring of $D_1,\ldots,D_s$ is an $\NN^s$-graded algebra of almost integral type (in the sense of \autoref{def_multigrad_integral}).
	For the single graded case, see \cite[Theorem 3.7]{KAVEH_KHOVANSKII}.
	
	\begin{proposition}
		\label{prop_sect_ring}
		$\SSS(D_1,\ldots,D_s)$ is an $\NN^s$-graded algebra of almost integral type.
	\end{proposition}
	\begin{proof}
		First, by Chow's lemma (see, e.g.~\cite[Theorem 13.100]{GORTZ_WEDHORN} or \cite[Exercise II.4.10]{HARTSHORNE}), there exists a proper birational morphism $\pi : X' \rightarrow X$ where $X'$ is a normal projective variety over $\kk$.
		Since for all $\bn \in \NN^s$ we have $\HH^0\big(X, \OO(\bn D)\big) \hookrightarrow \HH^0\big(X', \pi^*\OO(\bn D)\big)$, it suffices to assume that $X$ is a normal projective variety over $\kk$, and we do so.
		
		We can find a very ample divisor $H$ on $X$ such that $D_i \le H$, and so $\OO(\bn D) \subseteq \OO(|\bn|H)$ for all $\bn \in \NN^s$ (see, e.g.,~\cite[Example 1.2.10]{LAZARSFELD} and \cite[Theorem 3.9]{KAVEH_KHOVANSKII}).
		From the fact that the section ring $\SSS(H) = \bigoplus_{n=0}^\infty \HH^0(X, \OO(nH))$ is of integral type (see \cite[Exercise II.5.14]{HARTSHORNE}), we obtain that $\SSS(D_1,\ldots,D_s)$ is an $\NN^s$-graded algebra of almost integral type.
	\end{proof}

	Here, we study multigraded linear series as defined below. 
	
	\begin{definition}
		A \emph{multigraded linear series} associated to the divisors $D_1,\ldots,D_s$ is an $\NN^s$-graded $\kk$-subalgebra $W$ of the section ring $\SSS(D_1,\ldots,D_s)$. 
	\end{definition}
	
	Let $W \subseteq \SSS(D_1,\ldots,D_s)$ be a multigraded linear series and suppose that $[W]_{\ee_i} \neq 0$ for all $1 \le i \le s$.
	By \autoref{prop_sect_ring}, it follows that $W$ is an $\NN^s$-graded algebra of almost integral type. 
	The \emph{Kodaira-Itaka dimension} of $W$ is denoted and given by 
	$$
	\kappa(W) \; := \; \dim(W) - s,
	$$
	where as before $\dim(W)$ denotes the Krull dimension of the $\NN^s$-graded algebra $W$ of almost integral type.
	This value was the correct asymptotic for the volume function of an $\NN^s$-graded algebra of almost integral type (see \autoref{eq_function_A}). 
	Additionally, notice that this agrees with the definition used in \cite[Section 7]{cutkosky2014} for the case of singly graded linear series.
	Following \autoref{def_glob_NO}, let $\Delta(W)$ be the global Newton-Okounkov body of $W$.
	As in \autoref{def_uniform_A}, consider the integers $\ind(W)$ and $\ell_W$.
	
	Our main result regarding multigraded linear series is the theorem below, and it follows rather easily from our previous developments.
	
	\begin{theorem}
		\label{thm-linear-series}
		Assume \autoref{setup_linear_series}. 
		Let $W \subset \SSS(D_1,\ldots,D_s)$ be a multigraded linear series and suppose that $[W]_{\ee_i} \neq 0$ for all $1 \le i \le s$.				
		Then, the following statements hold: 
		\begin{enumerate}[(i)]
			\item The volume function 
			$$
			F_W(\bn) \; := \; \lim_{n \to \infty} \frac{\dim_{\kk}\big([W]_{n\bn}\big)}{n^{\kappa(W)}}
			$$
			of $W$ is well-defined for all $\bn \in \ZZ_+^s$.
			\item There exists a unique continuous function that is homogeneous of degree $\kappa(W)$ and log-concave and that extends the volume function $F_W(\bn)$ of part (i) to the positive orthant $\RR_{\geq 0}^s$.
			This function is given by 
			$$
			F_W \, :\, \RR_{\geq 0}^s\to \RR, \qquad x \,\mapsto\, \ell_W \cdot \frac{\Vol_{\kappa(W)}\big(\Delta(W)_x\big)}{\ind(W)}.
			$$
			\item If $W$ has a decomposable grading, then there exists a homogeneous polynomial $G_W(\bn) \in \RR[n_1,\ldots,n_s]$ of degree $\kappa(W)$ with non-negative real coefficients such that 
			$$
			F_W(\bn)  \;=\;  G_W(\bn)
			$$  for all $\bn \in \ZZ_+^s$.
		\end{enumerate}
	\end{theorem}
	\begin{proof}
		(i) The function is obtained as in \autoref{eq_function_A}.
		
		(ii) It follows from \autoref{cor-global}.
		
		(iii) It follows from \autoref{thm_poly_decomp}.
	\end{proof}

	For a closed subscheme  $Y \subset  \PP_\kk^{m_1} \times_\kk \cdots \times_\kk \PP_\kk^{m_s}$ of a multiprojective space over $\kk$, we can consider the \emph{multidegrees} of $Y$.
	These fundamental invariants go back to the work of van~der~Waerden \cite{VAN_DER_WAERDEN}.
	If $S$ is a standard $\NN^s$-graded algebra that coincides with the multihomogeneous coordinate ring of $Y$, then for each $\dd \in \NN^s$ with $\lvert \dd \rvert = \dim(Y)$, one way of defining the multidegree of $Y$ of type $\dd$ is by setting 
	$$
	\deg(\dd; Y) \; := \; e(\dd;S).
	$$
	The following theorem is an extension of \cite[Theorem 3.3]{KAVEH_KHOVANSKII} to a multigraded setting.
	It expresses the mixed multiplicities of a multigraded linear series with decomposable grading in terms of the multidegrees of the image of the corresponding Kodaira rational maps.
	
	\begin{theorem}
	    \label{thm_Kodaira_multigrad}
		Assume \autoref{setup_linear_series}. 
		Let $W \subset \SSS(D_1,\ldots,D_s)$ be a multigraded linear series with decomposable grading and suppose that $[W]_{\ee_i} \neq 0$ for all $1 \le i \le s$.	
		For each $p \ge 1$, consider the corresponding Kodaira rational map 
		$$
		\Pi_{[W]_{(p,\ldots,p)}} \; : \; X \;\dashrightarrow \; \PP_\kk^{\dim_{\kk}([W]_{p\ee_1}) - 1} \times_\kk \cdots \times_\kk \PP_\kk^{\dim_{\kk}([W]_{p\ee_s}) - 1}.
		$$			
		Let $Y_{[W]_{(p,\ldots,p)}}$ be the closure of the image of $\Pi_{[W]_{(p,\ldots,p)}}$.
		Then, for each  $\dd = (d_1,\ldots,d_s) \in \NN^s$ with $|\dd| = \kappa(W)$, we have the following equalities
		$$
		e(\dd; W) 
		\; = \; 
		\lim_{p \to \infty}  \frac{
			\deg\big(\dd;Y_{[W]_{(p,\ldots,p)}}\big)
		}{p^{\kappa(W)}}
		\; = \; 
		\sup_{p \in \ZZ_+}  \frac{
			\deg\big(\dd;Y_{[W]_{(p,\ldots,p)}}\big)
		}{p^{\kappa(W)}}.
		$$
	\end{theorem}
	\begin{proof}
		The result follows by applying \autoref{cor_vol_mult} to the multigraded linear series $W$ that has a decomposable grading.
	\end{proof}

	Finally, below we have an example where the section ring has a decomposable grading.

	\begin{example}
		Let $C$ be a smooth projective algebraic curve of genus $g$ over an algebraically closed field $\kk$.  Let $D_1,\ldots, D_s$ be divisors on $C$ with $\deg(D_i)\geq 2g+1$. Then by \cite[Theorem 6]{mumford}, the tensor product map
		\[
		\HH^0\big(X, \OO(n_1D_1)\big)\otimes \cdots \otimes \HH^0\big(X, \OO(n_sD_s)\big) \to   \HH^0\big(X, \OO(n_1D_1+\cdots+n_sD_s)\big)
		\]
		is surjective. Hence, the multigraded linear series $W=	\SSS(D_1,\ldots,D_s)$ has a decomposable grading and by part (iii) of \autoref{thm-linear-series} the function $F_W$ is a homogeneous polynomial of degree $\kappa(W)$. 
		
		Finally, we would like to remark that the conditions on $C$ and $D_1,\ldots, D_s$ can be relaxed slightly by considering generalizations of Mumford's theorem, see \cite{green1986projective, eisenbud1988determinantal, butler1999global} for details.
	\end{example}

\section*{Acknowledgments}	
The authors thank the reviewer for helpful comments and suggestions.
F.M. was partially supported by the grants G023721N and G0F5921N (Odysseus programme) from the Research Foundation - Flanders (FWO), the KU Leuven grant iBOF/23/064 and the UiT Aurora project MASCOT.

	\bibliographystyle{amsplain} 
	\bibliography{references.bib}

\end{document}